\documentclass{article}
\usepackage[utf8]{inputenc}
\usepackage{eufrak} % always place before amsfonts
\usepackage{amsfonts}
\usepackage{graphicx}
\usepackage{graphics}
\usepackage{adjustbox}
\usepackage{epstopdf}
\usepackage[numbers,compress,sort]{natbib} %,sort
\usepackage{algorithmic}
\usepackage{booktabs}
\ifpdf
  \DeclareGraphicsExtensions{.eps,.pdf,.png,.jpg}
\else
  \DeclareGraphicsExtensions{.eps}
\fi

\usepackage{amsopn}

\usepackage{comment}

\usepackage[breaklinks=true, pdftex, 
pdfborder={0 0 0}, colorlinks=true, linkcolor=blue, citecolor=blue, urlcolor=blue, 
bookmarks=false, pdfpagelabels=false]{hyperref}

\usepackage{fullpage}
\usepackage{times}
\usepackage{ifthen}   % for conditionals
\usepackage{setspace}
\usepackage{amssymb,amsmath,amsthm}

\usepackage{doi}
\usepackage{dsfont}
\usepackage{mathrsfs} % for \mathscr
\usepackage{color}
\usepackage{wrapfig}
\usepackage{multirow}
\usepackage{multicol}
\usepackage{pdfpages}
\usepackage{longtable} 	
\usepackage{datetime}
\usepackage{xspace}
\usepackage{paralist} %% This gives enumeration environments & compactenum 
\usepackage{enumitem} % For tighter lists

\usepackage{listings}
\usepackage{upgreek}
\usepackage[bbgreekl]{mathbbol}
\DeclareSymbolFontAlphabet{\mathbb}{AMSb}
\DeclareSymbolFontAlphabet{\mathbbl}{bbold}

\definecolor{codegreen}{rgb}{0,0.6,0}
\definecolor{codegray}{rgb}{0.5,0.5,0.5}
\definecolor{codepurple}{rgb}{0.58,0,0.82}
\definecolor{backcolour}{rgb}{0.95,0.95,0.92}
\definecolor{ForestGreen}{RGB}{34,139,34}

\usepackage{xcolor}
\usepackage{listings}

\usepackage{xparse}

\usepackage{stackengine}

\NewDocumentCommand{\codeword}{v}{%
	\texttt{\textcolor{blue}{#1}}%
}

\lstdefinestyle{mystyle}{
	backgroundcolor=\color{backcolour},   
	commentstyle=\color{codegreen},
	keywordstyle=\color{magenta},
	numberstyle=\tiny\color{codegray},
	stringstyle=\color{codepurple},
	basicstyle=\ttfamily\footnotesize,
	breakatwhitespace=false,         
	breaklines=true,                 
	captionpos=b,                    
	keepspaces=true,                 
	numbers=left,                    
	numbersep=5pt,                  
	showspaces=false,                
	showstringspaces=false,
	showtabs=false,                  
	tabsize=2
}

\usepackage{titlesec}
\usepackage{mdframed}  % creating color text boxes
\usepackage{colortbl}
\usepackage{mathrsfs} % for \mathscr
\definecolor{highlight}{HTML}{DFEFF9}
\mdfdefinestyle{highlightbox}{leftmargin=0pt,rightmargin=0pt,%
	innerleftmargin=3pt,innerrightmargin=2pt,%
	innertopmargin=4pt,innerbottommargin=4pt,%
	backgroundcolor=highlight,linecolor=white}
\definecolor{algm}{HTML}{D4EFDF} % background color
\mdfdefinestyle{highlightalg}{leftmargin=0pt,rightmargin=0pt,%
	innerleftmargin=3pt,innerrightmargin=2pt,%
	innertopmargin=4pt,innerbottommargin=4pt,%
	backgroundcolor=algm,linecolor=white}
\definecolor{soft}{HTML}{F2D7D5} % background color
\mdfdefinestyle{highlightstruct}{leftmargin=0pt,rightmargin=0pt,%
	innerleftmargin=3pt,innerrightmargin=2pt,%
	innertopmargin=4pt,innerbottommargin=4pt,%
	backgroundcolor=soft,linecolor=white}

\usepackage{soul}

\makeatletter
\def\namedlabel#1#2{\begingroup
	#2%
	\def\@currentlabel{#2}%
	\phantomsection\label{#1}\endgroup
}
\makeatother

\newcommand{\Reals}{\mathbb{R}} % Real numbers
\newcommand{\R}{\Reals} % Real numbers x2
\DeclareMathOperator*{\argmin}{arg\,min}

\newcommand{\E}[2][]{\operatorname{\mathbb{E}}_{#1}\left[ #2\right]} % Exp value
 
\usepackage{bm} % For bold math symbols

\newcommand{\ab}{\bm{a}}

\newcommand{\ub}{\bm{u}}

\newtheorem{definition}{Definition}[section]
\newtheorem{theorem}{Theorem}[section]
\newtheorem{lemma}{Lemma}[section]
\newtheorem{proposition}{Proposition}[section]

\newtheorem{corollary}{Corollary}[section]

\newtheorem{remark}{Remark}[section]
\usepackage{tikz}
\usetikzlibrary{positioning}
\usetikzlibrary{calc}
\usetikzlibrary{shapes.geometric}
\usetikzlibrary{topaths}
\usetikzlibrary{external}
\usetikzlibrary{fadings,decorations.pathreplacing} 

\usepackage{pgfplots}
\usetikzlibrary{shapes,decorations.pathmorphing,patterns}
\pgfplotsset{compat=newest}
\usepgfplotslibrary{fillbetween}

\graphicspath{{figures/}}
\usepackage{hyperref}
\hypersetup{
unicode = false,
pdftoolbar = true,
pdfmenubar = true,
pdffitwindow = true,
pdftitle = {Sparse-JLTs},
pdfauthor = {Dzahini},
pdfsubject = {RandomMatrix},
pdfnewwindow = true,
pdfkeywords = {Random, Matrix},
colorlinks = true,
linkcolor = blue,
citecolor = blue,
filecolor = black,
urlcolor = blue,
breaklinks = true
}

\newcommand{\snOne}{\mathbb{S}^{n-1}}
\newcommand{\sNOne}{\mathbb{S}^{N-1}}
\newcommand{\N}{\mathbb{N}}

\newcommand{\F}{\mathcal{F}}

\newcommand{\n}{\mathcal{N}}

\newcommand{\pr}{\mathbb{P}}
\newcommand{\M}{\mathcal{M}}

\newcommand{\ds}{\displaystyle}

\newcommand{\abs}[1]{\left\lvert#1\right\rvert}
\newcommand{\intbracket}[1]{\left[\!\left[#1\right]\!\right]}
\newcommand{\accolade}[1]{\left\lbrace#1\right\rbrace}

\newcommand{\prob}[1]{\mathbb{P}\left(#1\right)}
\newcommand{\var}[1]{\mathbb{V}\left(#1\right)}

\newcommand{\norme}[1]{\left\lVert#1\right\rVert}
\newcommand{\scal}[2]{\left\langle#1,#2\right\rangle}

\newcommand{\normii}[1]{{\left\lVert#1\right\rVert}_{2}}
\newcommand{\normiin}[1]{{\left\lVert#1\right\rVert}_{2,n}}
\newcommand{\normiiN}[1]{{\left\lVert#1\right\rVert}_{2,N}}

\newcommand{\norminf}[1]{{\left\lVert#1\right\rVert}_{\infty}}

\newcommand{\comb}[2]{\begin{pmatrix}
{#1}\\
{#2}
\end{pmatrix}}

\newcommand{\rn}{\mathbb{R}^n}
\newcommand{\rN}{\mathbb{R}^N}

\newcommand{\rNn}{\mathbb{R}^{N\times n}}
\newcommand{\rnN}{\mathbb{R}^{n\times N}}

\newcommand{\vi}{\bm{v}}

\newcommand{\Hmax}{\mathfrak{H}_{\max}}

\newcommand{\x}{\bm{x}}
\newcommand{\y}{\bm{y}}

\newcommand{\Amatrix}{\mathfrak{A}}

\newcommand{\aij}{a_{ij}}

\newcommand{\Xij}{X_{ij}}
\newcommand{\Rij}{R_{ij}}
\newcommand{\Yij}{Y_{ij}}
\newcommand{\dirac}[1]{\bbdelta_{#1}}
\newcommand{\normPsiTwo}[1]{{\left\lVert#1\right\rVert}_{\psi_2}}

\newcommand{\Hfrac}{\mathfrak{H}}
\newcommand{\Zij}{Z_{ij}}
\newcommand{\Hij}{\mathfrak{H}_{ij}}
\newcommand{\Uij}{U_{ij}}
\newcommand{\bl}{\color{black}}
\newcommand{\bbl}{\color{black}}

\usepackage{orcidlink}
\title{A class of sparse Johnson--Lindenstrauss transforms and analysis of their extreme singular values\thanks{To appear in the {\it SIAM Journal on Matrix Analysis and Applications}. This work was supported in part by the U.S.~Department of Energy, Office of Science, Office of Advanced Scientific Computing Research Applied Mathematics Program under Contract Nos.\ DE-AC02-05CH11231 and DE-AC02-06CH11357.}}
\author{
	\href{mailto:kdzahini@anl.gov}{K. J. Dzahini\orcidlink{0000-0002-1515-4251}}\thanks{Argonne National Laboratory, 9700 S. Cass Avenue, Lemont, IL 60439, USA (\href{https://www.anl.gov/profile/kwassi-joseph-dzahini}{www.anl.gov/profile/kwassi-joseph-dzahini}).
	}
	\and
	\href{mailto:wild@lbl.gov}{S. M. Wild\orcidlink{0000-0002-6099-2772}}\thanks{Lawrence Berkeley National Laboratory, 1 Cyclotron Road, Berkeley, CA 94720, USA ( \href{https://wildsm.github.io/}{wildsm.github.io/}).
	}
}
\date{\today}
\begin{document}

%\linenumbers % To make it easier for reviewers

\maketitle

%\vspace*{-0.5cm}

\noindent
{\bf Abstract:} The Johnson--Lindenstrauss (JL) lemma is a powerful tool for dimensionality reduction in modern algorithm design. The lemma states that any set of high-dimensional points in a Euclidean space can be {\bbl projected into} lower dimensions while approximately preserving pairwise Euclidean distances. Random matrices satisfying this lemma are called JL transforms (JLTs). Inspired by existing $s$-hashing JLTs with exactly $s$ nonzero elements on each column, {\bl the present work introduces an ensemble of sparse matrices encompassing so-called} $s$-hashing-like matrices whose expected number of nonzero elements on each column is~$s$. {\bl The} independence of the sub-Gaussian entries of {\bl these} matrices and the knowledge of their exact distribution play an important role in their analyses. Using properties of independent sub-Gaussian random variables, these matrices are demonstrated to be JLTs, and their {\bbl (nontrivial)} smallest{\footnote{\bbl Throughout the manuscript, ``smallest singular value'' refers to the smallest nontrivial (that is, not trivially zero) singular value of a given matrix.}} and largest singular values are estimated non-asymptotically using a technique from geometric functional analysis. {\bl As the dimensions of the matrix grow to infinity,} these singular values are proved to converge almost surely to fixed quantities {\bl (by using the universal Bai--Yin law), and in distribution to the Gaussian orthogonal ensemble (GOE) Tracy--Widom law after proper rescalings. Understanding the behaviors of extreme singular values is important in general because they are often used to define} a measure of stability of matrix algorithms. For example, JLTs were recently used in derivative-free optimization algorithmic frameworks to select random subspaces in which are constructed random models or poll directions to achieve scalability, whence estimating their smallest singular value in particular helps determine the dimension of these subspaces. 

$ $

$ $
\clearpage

\section{Introduction}
% \mathfrak{H}=\frac{1}{\sqrt{n}}(\Xij)_{\underset{1\leq j\leq N}{1\leq i\leq n}}
Johnson--Lindenstrauss transforms (JLTs; see Definition~\ref{JLTransforms} below) are random matrices satisfying the following celebrated result~\cite[Lemma~1.1]{KaNel2014SparseLidenstrauss} due to Johnson and Lindenstrauss (JL)~\cite{JohnsonLind1984}.
\begin{lemma}\label{JLLemma}(Johnson--Lindenstrauss).
For any integer \scalebox{0.88}{$N>0$} and \scalebox{0.88}{$0<\epsilon, \delta<1/2$}, there exists a probability distribution on \scalebox{0.88}{$n\times N$} real matrices for \scalebox{0.88}{$n=\Theta(\epsilon^{-2}\log(1/\delta))$} such that \scalebox{0.88}{$\prob{(1-\epsilon)\normiiN{\x}\leq \normiin{\mathfrak{H}\x}\leq (1+\epsilon)\normiiN{\x}}\geq 1-\delta$}, for any \scalebox{0.88}{$\x\in \rN$} and {\bbl any matrix~$\mathfrak{H}$ drawn from the aforementioned distribution.}
\end{lemma}
\vspace*{-0.4cm}
$ $\\As mentioned in~\cite{KaNel2014SparseLidenstrauss}, Lemma~\ref{JLLemma} is ``a key ingredient in the JL flattening theorem'', which states that an \scalebox{0.88}{$m$}-point set in any Euclidean space can be mapped to \scalebox{0.88}{$\mathcal{O}(\epsilon^{-2}\ln(m))$} dimensions so that all pairwise Euclidean distances are preserved up to a multiplicative factor between \scalebox{0.88}{$1-\epsilon$} and \scalebox{0.88}{$1+\epsilon$}. Moreover, the JL lemma, as a useful tool for linear regression and low-rank approximation, is also used to reduce the amount of storage that is necessary to store a dataset in streaming algorithms and to speed up solutions to high-dimensional problems and can further be used to speed up some clustering and string-processing algorithms~\cite{KaNel2014SparseLidenstrauss}. JLTs were also used recently in derivative-free optimization~\cite{AuHa2017,CoScVibook} to achieve scalability by means of random models~\cite{CarFowSha2022Randomises,CRsubspace2021,DzaWildSub2022,shao2022Thesis} or polling directions~\cite{RobRoy2022RedSpace} constructed in random subspaces.

Let \scalebox{0.88}{$q\geq 2$, $\hat{q}:=\sqrt{\frac{q}{2}}$}, and let \scalebox{0.88}{$X$} be a random variable distributed according to 
%\scalebox{0.88}{$\bbmu_X(q):=\frac{1}{q}\dirac{-\hat{q}}+\left(1-\frac{2}{q}\right)\dirac{0}+\frac{1}{q}\dirac{\hat{q}}$}, where \scalebox{0.88}{$\dirac{\nu}$} denotes the Dirac measure~\cite[p.~11]{tao2012topicsInRMT} centered on~\scalebox{0.88}{$\nu$}
{\bbl $\bbmu_X(q)$, a formal definition of which is provided in~\eqref{muXqmuYq}, 
where \scalebox{0.88}{$X\sim \bbmu_X(q)$} means} \scalebox{0.88}{$\prob{X= -\hat{q}}=\prob{X= \hat{q}} = \frac{1}{q}$} and \scalebox{0.88}{$\prob{X=0}=1-\frac{2}{q}$}. Consider the matrix $\mathfrak{H}$ defined by \scalebox{0.88}{$\mathfrak{H}=\frac{1}{\sqrt{n}}(\Xij)_{\underset{1\leq j\leq N}{1\leq i\leq n}}\in \rnN$} with \scalebox{0.88}{$N\geq n$}, that is, whose entries are independent copies of \scalebox{0.88}{$\frac{1}{\sqrt{n}}X$}, where \scalebox{0.88}{$X\sim \bbmu_X(q)$}. When \scalebox{0.88}{$q=\frac{2n}{s}$} for some \scalebox{0.88}{$s\in (0, n]$}, then $s$ is the expected number of nonzero elements per column of \scalebox{0.88}{$\Hfrac$}, which can be rewritten as \scalebox{0.88}{$\mathfrak{H}=\frac{1}{\sqrt{s}}(\Yij)_{\underset{1\leq j\leq N}{1\leq i\leq n}}$} (referred to as an $s${\it -hashing-like matrix}), as will be detailed later, where the \scalebox{0.88}{$\Yij\in \accolade{-1,0,1}$} are independent copies of \scalebox{0.88}{$Y$}
%\scalebox{0.88}{$Y\sim \frac{s}{2n}\dirac{-1}+\left(1-\frac{s}{n}\right)\dirac{0}+\frac{s}{2n}\dirac{1}$}
{\bbl distributed according to \scalebox{0.88}{$\prob{Y= -1}=\prob{Y= 1}=\frac{s}{2n}$} and \scalebox{0.88}{$\prob{Y=0}=1-\frac{s}{n}$}.}
%thus generalizing the projection matrices $\Hfrac$ considered by Achlioptas~\cite{Database2003Ach}  defined as above, which are demonstrated to be JLTs when $q\in\accolade{2, 6}$. \\
\vspace*{0.15cm}\\
{\textbf{Main goal.}} The goal of this manuscript is twofold. The first one is to derive sufficient conditions on $n${\bl , \scalebox{0.88}{$q$}} and $s$ so that \scalebox{0.88}{$\Hfrac$} is a JLT. The second is to perform asymptotic and non-asymptotic analyses of the {\it extreme singular values} of \scalebox{0.88}{$\Hfrac$}, that is, the largest \scalebox{0.88}{$s_1(\Hfrac)$} (the {\it soft edge}), and the smallest {\bbl or~$n$th largest} \scalebox{0.88}{$s_n(\Hfrac)$} (the {\it hard edge}~\cite{VershIntroAsym2010}).
\vspace*{0.15cm}\\
{\textbf{Brief overview of random matrix theory.}} Random matrix theory focuses on properties of matrices selected from some distribution over the set of all random matrices of the same dimensions~\cite{VershIntroAsym2010}. Since its inception, it has  been preoccupied mostly with the asymptotic properties of these matrices as their dimensions tend to infinity~\cite{RudVersh2010ExtrSingVal}. Foundational examples include  Wigner's semicircle law~\cite{WignerSemiCircle1958} for the empirical measures of eigenvalues of symmetric Gaussian matrices; the Marchenko--Pastur law~\cite{MarchenKoPastur1967} for Wishart matrices; and the Tracy--Widom~\cite{felSodTracyWid2010,SoshnikovTracyWid2002} and Bai--Yin~\cite{BaiYinSmallCovEig1993} laws for the extreme eigenvalues of Wishart matrices. Asymptotic random matrix theory is well suited for many purposes in statistical physics where random matrices are used to model infinite-dimensional operators~\cite{VershIntroAsym2010}, and it recently found several applications in large-scale optimization (see, e.g.,~\cite{PaqLeePedPaqSGDlarge2021,PaqMerPedHalting2022} and references therein). However, despite the incredibly precise predictions of asymptotic random matrix theory as dimensions tend to infinity, the limiting regime may not be very useful in many other areas such as compressed sensing and some areas of optimization, where one is interested in what happens in fixed dimensions, that is, in the {\it non-asymptotic} regime.

Results from non-asymptotic theory of random matrices are in demand in several of today's applications~\cite{RudVersh2010ExtrSingVal}, and 
% ~\cite{CRsubspace2021,DzaWildSub2022,RobRoy2022RedSpace}
% Non-asymptotic results on random matrices are in demand in a number of today’s applications that operate in high but ﬁxed dimensions
outstanding growth in the use of JLTs 
%Johnson--Lindenstrauss Transforms (JLTs, see Definition~\ref{JLTransforms} below) 
has recently attracted interest in myriad fields for solving large-scale optimization problems. 
%GAIL - interesting since earlier you said may not be useful in some areas of optimization. Is that because of fixed vs independent?
%Such matrices satisfy the result of the following celebrated lemma (see~\cite[Lemma~1.1]{KaNel2014SparseLidenstrauss}) due to Johnson and Lindenstrauss 
%(JL)~\cite{JohnsonLind1984}.
In~\cite{CRsubspace2021,DzaWildSub2022}, JLTs required to satisfy\footnote{\bbl Here \scalebox{0.88}{$\norme{\cdot}$} denotes any matrix norm (such as the Frobenius and spectral matrix norms), which is {\it consistent} with the Euclidean norm~\scalebox{0.88}{$\normii{\cdot}$}; that is \scalebox{0.88}{$\normii{\mathfrak{H}\x}\leq \norme{\mathfrak{H}}\normii{\x}$}.} \scalebox{0.88}{$\norme{\mathfrak{H}}\leq \Hmax$}, for some constant \scalebox{0.88}{$\Hmax>0$}, were used to construct low-dimensional random subspaces whose dimensions are independent of those of the underlying problem and in which are built random models used to achieve scalability. 

Besides the study of JLTs, another area where significant progress was recently made is that of the {\it non-asymptotic theory of extreme singular values} of random matrices \scalebox{0.88}{$\Hfrac$}. In fact, once \scalebox{0.88}{$\Hfrac$} is viewed as a linear operator, one may first want to control its magnitude by placing useful lower and upper bounds on it. As highlighted in~\cite[Section~2]{RudVersh2010ExtrSingVal}, such bounds are conveniently provided by the smallest {\bbl (or~$n$th largest)} singular value \scalebox{0.88}{$s_n(\Hfrac)$} and the largest singular value \scalebox{0.88}{$s_1(\Hfrac)$}. Understanding the behavior of extreme singular values of random matrices is crucial in a number of fields such as in statistics when analyzing datasets with a large but fixed number of parameters{\bl , and} in geometric functional analysis where random operators on finite-dimensional spaces {\bl are used.}
%, and in numerical linear algebra where the  condition number \scalebox{0.88}{$\kappa(\Hfrac):=s_1(\Hfrac)/s_n(\Hfrac)$} is often used as a measure of the stability of matrix algorithms. 
In a randomized direct-search algorithmic framework recently introduced in~\cite{RobRoy2022RedSpace}, JLTs \scalebox{0.88}{$\Hfrac\in\rnN$, $N\geq n$}, were used for the selection of random subspaces in which are defined  polling directions, and their smallest {\bbl (or~$n$th largest)} singular values were required to satisfy the condition \scalebox{0.88}{$s_n(\Hfrac)\geq \upphi$} (in addition to \scalebox{0.88}{$\norme{\mathfrak{H}}\leq \Hmax$} as above) with high probability for some constant \scalebox{0.88}{$\upphi>0$}. We remark that in all the algorithmic frameworks~\cite{CRsubspace2021,DzaWildSub2022,RobRoy2022RedSpace} discussed above, the random subspaces are defined as the range of \scalebox{0.88}{$\Hfrac^{\top}\in\rNn$} so that determining their dimension requires estimating \scalebox{0.88}{$s_n(\Hfrac)$} non-asymptotically.
\vspace*{0.15cm}\\
{\textbf{Related works.}} Many {\bbl variants} of the JL lemma exist. Indeed, after the original {\bbl result of} Johnson and Lindenstrauss~\cite{JohnsonLind1984}, Indyk and Motwani~\cite{IndykMotwani1998} (see also~~\cite[Theorem~2.1]{WoodrDavP2014} and~\cite[Theorem~3.1]{DzaWildSub2022}) provided a {\bbl variant} by showing that the entries of the projection matrix \scalebox{0.88}{$\Hfrac$} need only  be independent Gaussian \scalebox{0.88}{$\n(0, 1/n)$} random variables for appropriate values of \scalebox{0.88}{$n$}. Dasgupta and Gupta~\cite{DasguptaGupta2003JLT} simplified the proof of Indyk and Motwani using  moment-generating functions. After proving that the spherical symmetry of the projection coefficients on which  the results in~\cite{DasguptaGupta2003JLT,IndykMotwani1998} rely is not necessary to obtain JLTs that maintain \scalebox{0.88}{$\epsilon$}-distortion, Achlioptas~\cite{Database2003Ach} showed that concentration of the projected points is sufficient and chose the entries of $\Hfrac$ to be independent copies of \scalebox{0.88}{$\frac{1}{\sqrt{n}}X$}, where \scalebox{0.88}{$X\sim \bbmu_X(q)$} as above, with \scalebox{0.88}{$q\in\accolade{2, 6}$}. 
In that way, the expectation of the number~{\bbl $\mathscr{N}$} of nonzero elements on each column of \scalebox{0.88}{$\Hfrac$} is either \scalebox{0.88}{$s:=\E{{\bbl \mathscr{N}}}=n$} when \scalebox{0.88}{$q=2$} or \scalebox{0.88}{$s=n/3$} when \scalebox{0.88}{$q=6$}.  Matou\v{s}ek~\cite{MatJLT2008Variants} {\bbl proved} a generalized variant of the JL lemma (reported as Theorem~\ref{Matousek31} below) using properties of sub-Gaussian random variables {\bbl by considering} in particular the above projection matrix using \scalebox{0.88}{$q=2/C_0\upalpha^2\ln(N/\epsilon\delta)\geq 2$}, where \scalebox{0.88}{$\upalpha\in \left[\frac{1}{\sqrt{N}}, 1\right]$} and \scalebox{0.88}{$C_0$} is sufficiently large (see \cite[Theorem~4.2]{MatJLT2008Variants} or~\cite[Theorem~2]{FedSchJohHeo2018}, reported as Theorem~\ref{MatTheorem42} below). We note, however,  that the latter result holds only for unit vectors \scalebox{0.88}{$\x\in \rN$} with \scalebox{0.88}{$\norminf{\x}\leq \upalpha$}, that is, for \scalebox{0.88}{$\x$} with sufficiently small \scalebox{0.88}{$\ell_{\infty}$}-norm. On the other hand, using graph- and code-based constructions, Kane and Nelson proposed in~\cite{KaNel2014SparseLidenstrauss} JLTs in the spirit of those of Achlioptas but requiring that~$s$ be instead the exact number of nonzero elements of the form \scalebox{0.88}{$\frac{1}{\sqrt{s}}X$} on each column of \scalebox{0.88}{$\Hfrac$}, where \scalebox{0.88}{$X\sim\bbmu_X(2)$}, that is, a \scalebox{0.88}{$(\pm 1)$} or signed Bernoulli random variable. Consequently, they obtained a better asymptotic bound on the sparsity of the resulting \scalebox{0.88}{$\Hfrac$}, referred to as $s${\it -hashing matrices} in the present manuscript and~\cite{CRsubspace2021,DzaWildSub2022,RobRoy2022RedSpace}.

While the above Gaussian approach by Indyk and Motwani is among the simplest for obtaining JLTs, unfortunately the inequality \scalebox{0.88}{$\norme{\mathfrak{H}}\leq \Hmax$} required in~\cite{CRsubspace2021,DzaWildSub2022} {\bbl to hold with probability one} is satisfied for Gaussian matrices only with high probability for \scalebox{0.88}{$\Hmax=\Theta(\sqrt{N/n})$}~\cite[Corollary~3.11]{bandeira2016sharp}. %On the other hand as emphasized in~\cite{FedSchJohHeo2018}. 
On the other hand, even though for the aforementioned $s$-hashing matrices \scalebox{0.88}{$\norme{\mathfrak{H}}\leq\Hmax:=\sqrt{N}$} {\bbl with probability one},
to the best of our knowledge and as highlighted in~\cite{RobRoy2022RedSpace}, no  theoretical results exist that are specific to the extreme singular values~\scalebox{0.88}{$s_n(\mathfrak{H})$} and~\scalebox{0.88}{$s_1(\mathfrak{H})$} of $s$-hashing matrices, whose entries on each column are sub-Gaussian~\cite[Definition~2.5.6]{vershynin2018HDP}, although non-independent. The non-asymptotic behaviors of extreme singular values of random matrices with sub-Gaussian entries have been extensively studied~\cite{LitPajRud2005Pol,Rudelson2014LectNotes,RudVersh2008Lit,RudVersh2010ExtrSingVal,RudelsSmallRect2009,VershIntroAsym2010,vershynin2018HDP} with powerful techniques. Even if some of these techniques are intended for the study of matrices with independent sub-Gaussian rows~\cite[Theorems~5.39 and~5.58]{VershIntroAsym2010}-\cite[Theorem~4.6.1]{vershynin2018HDP} as is the case for the transpose of $s$-hashing matrices,  most of them focus on matrices with independent sub-Gaussian entries. Furthermore, the aforementioned rows were required to be {\it isotropic}~\cite[Definition~3.2.1]{vershynin2018HDP}, which seems not to be the case for these $s$-hashing matrices whose distribution of each column is to our knowledge a priori unknown.
\vspace*{0.15cm}\\
{\textbf{Our Contributions.}}
For any fixed \scalebox{0.88}{$q\geq 2$}, by considering the above random matrices \scalebox{0.88}{$\mathfrak{H}=\frac{1}{\sqrt{n}}(\Xij)_{\underset{1\leq j\leq N}{1\leq i\leq n}}\in \rnN$} with \scalebox{0.88}{$X\sim\bbmu_X(q)$}, therefore generalizing those of Achlioptas~\cite{Database2003Ach}, the present work shows in Theorem~\ref{FirstJLT1} that given \scalebox{0.88}{$\epsilon\in (0,1/2], \delta\in (0,1)$, $\Hfrac$} is a JLT provided that \scalebox{0.88}{$n\geq {\bl {\bbl C_{1,q}}}\epsilon^{-2}\ln(1/\delta)$}, {\bl  for a constant~\scalebox{0.88}{${\bbl C_{1,q}}$} depending only on~$q$}. Most important, this result holds not only for unit vectors $\x$ with sufficiently small \scalebox{0.88}{$\ell_{\infty}$}-norm or for specific \scalebox{0.88}{$q=2/C_0\upalpha^2\ln(N/\epsilon\delta)\geq 2$} as obtained by Matou\v{s}ek~\cite{MatJLT2008Variants} but for any vector \scalebox{0.88}{$\x\in \rN$}. Moreover, this work introduces $s$-hashing-like matrices \scalebox{0.88}{$\mathfrak{H}=\frac{1}{\sqrt{s}}(\Yij)_{\underset{1\leq j\leq N}{1\leq i\leq n}}$} with independent entries, where $s$ is the expected number of nonzero elements of the form \scalebox{0.88}{$\pm \frac{1}{\sqrt{s}}$} on each column of \scalebox{0.88}{$\Hfrac$}, and shows in Theorem~\ref{sTheorem} how $s$ and $n$ can be chosen so that \scalebox{0.88}{$\Hfrac$} is a JLT, a result in the spirit of that by Kane and Nelson~\cite{KaNel2014SparseLidenstrauss} for $s$-hashing matrices reported as~\cite[Theorem~3.2]{DzaWildSub2022}.

On the other hand, analyses of the  extreme singular values of the aforementioned $s$-hashing-like matrices have been performed. In the asymptotic case, using the Bai--Yin law~\cite{BaiYinSmallCovEig1993}, it has been demonstrated in Corollar{\bl ies}~\ref{CorBaiYin} {\bl and~\ref{CorBaiYin2}} that when \scalebox{0.88}{$N, n\to \infty$} while the  aspect ratio \scalebox{0.88}{$n/N\to\mathfrak{y}\in (0,1)$}, then \scalebox{0.88}{$s_n(\Hfrac)$} and \scalebox{0.88}{$s_1(\Hfrac)$} almost surely converge to \scalebox{0.88}{$\frac{1}{\sqrt{\mathfrak{y}}}-1$} and \scalebox{0.88}{$\frac{1}{\sqrt{\mathfrak{y}}}+1$}, {\bl respectively. Moreover, the latter results combined with existing ones from~\cite{felSodTracyWid2010,Peche2009Universality,SoshnikovTracyWid2002}, yield in Corollaries~\ref{s1Distribution} and~\ref{snDistribution} the convergence in distribution of \scalebox{0.88}{$s_1(\Hfrac)$} and \scalebox{0.88}{$s_n(\Hfrac)$} to the {\bl Gaussian orthogonal ensemble} (GOE) Tracy--Widom law.}  
%provided \scalebox{0.88}{$\frac{s}{n}=\frac{s(n)}{n}$} remains constant. 
These singular values are then estimated non-asymptotically through Theorems~\ref{LargestSingResult} and~\ref{SmallestSingResult} {\bl and Corollary~\ref{LargestSingCor}.} 
%in terms of explicit constants depending on~\scalebox{0.88}{$q, s, n$} or~\scalebox{0.88}{$N$}, as is also the case throughout the manuscript, unlike many prior works where constants are usually absolute but unknown. 
\vspace*{0.08cm}\\
{\textbf{Organization.}} The manuscript is organized as follows. Section~\ref{sec2} introduces {\bl an ensemble of sparse random matrices encompassing so-called} hashing-like matrices and provides sufficient conditions under which they are JLTs. Section~\ref{sec3} focuses on the {\bl asymptotic and non-asymptotic} analyses of the extreme singular values of these matrices, followed by a discussion and suggestions for future work.

%$ $\\
%\textcolor{red}{Todo}
%\begin{itemize}
%\item Literature review for RMT results, in particular sparse matrices.
%\item Focus on the fact that existing techniques almost all use independence of entries.
%\item Say that two-sided bounds techniques exist (Vershynin book), using independent and isotropic rows and motivate why this is not handled here for $\mathfrak{Q}^{\top}$... but future work: indeed, isotropy should be demonstrated!!
%\item Then say why motivated to assume independence.
%\item Wonder if in that independence case, one still has JLTs! 
%\item Briefly describe the model analyzed in this manuscript by means of its entries and mention that its construction will be detailed later in the next section. 
%\item Then specify that it still satisfies $\norme{\mathfrak{Q}}\leq \Hmax$ with $\Hmax\ \textcolor{blue}{=}\ \sqrt{N}$  !!
%\item Present the goals of the paper, what motivated them, the organization of the manuscript, etc.
%\end{itemize}

\section{Sparse random matrices as Johnson--Lindenstrauss transforms}\label{sec2}
In this section we introduce {\bl a general ensemble of sparse matrices encompassing} hashing-like matrices (see~\eqref{HMatrixEnsemble} and~\eqref{HashlikeEnsemble}) and derive sufficient conditions for {\bl these} matrices to be Johnson--Lindenstrauss transforms, using properties of sub-Gaussian variables.

All the random variables considered are defined on the same probability space \scalebox{0.88}{$(\Omega, \F, \pr)$}, where the  sample space $\Omega$ is a nonempty set with elements \scalebox{0.88}{$\omega$} referred to as  sample points and whose subsets are called {\it events}. \scalebox{0.88}{$\F$} is a \scalebox{0.88}{$\sigma$}-{\it algebra} consisting of a collection of these events, while the finite measure \scalebox{0.88}{$\pr$} defined on the measurable space \scalebox{0.88}{$(\Omega, \F)$} and satisfying \scalebox{0.88}{$\pr(\Omega)=1$} is referred to as the  probability measure. We denote by \scalebox{0.88}{$\mathscr{B}(\rn)$} the Borel \scalebox{0.88}{$\sigma$}-algebra of \scalebox{0.88}{$\rn$}, that is, the one generated by open sets of $\rn$. An \scalebox{0.88}{$\rn$}-valued random variable \scalebox{0.88}{$S$} is a measurable map from \scalebox{0.88}{$(\Omega, \F)$} into \scalebox{0.88}{$(\rn, \mathscr{B}(\rn))$}, where measurability means that \scalebox{0.88}{$\accolade{S\in \mathscr{E}}:=\accolade{\omega\in\Omega:S(\omega)\in \mathscr{E}}\in\F$} for all \scalebox{0.88}{$\mathscr{E}\in \mathscr{B}(\rn)$}. A real-valued  random matrix is a matrix-valued random variable, that is, a matrix in which all elements are \scalebox{0.88}{$\R$}-valued random variables.
The  Euclidean norm of an \scalebox{0.88}{$m$}-dimensional vector \scalebox{0.88}{$\x:=(x_1,\dots,x_m)^{\top}$} is denoted by \scalebox{0.88}{$\norme{\x}_{2,m}:=\left(\sum_{\ell=1}^{m}x_{\ell}^2\right)^{1/2}$}, while the norm of an operator or a matrix is denoted by \scalebox{0.88}{$\norme{\cdot}$}. {\bbl \scalebox{0.88}{$\mathbb{S}^{m-1}:=\accolade{\x\in\R^m:\norme{\x}_{2,m}=1}$} denotes the unit sphere of $\R^m$.} The indicator function of a set \scalebox{0.88}{$\mathscr{E}$} is denoted by \scalebox{0.88}{$\mathds{1}_{\mathscr{E}}$}. Given \scalebox{0.88}{$a, b\in \R$} with \scalebox{0.88}{$a\leq b$}, \scalebox{0.88}{$\intbracket{a, b}:=[a, b]\cap \mathbb{Z}$}, where \scalebox{0.88}{$\mathbb{Z}$} is the set of integers.

\subsection{The hashing-like matrix ensemble}
The construction of hashing-like {\bl matrices detailed below}, is inspired by the {\bl following} definition of {\it hashing matrices}~\cite{KaNel2014SparseLidenstrauss}. 

\begin{definition}\label{sHashingDef}($s$-hashing matrix).
Let \scalebox{0.88}{$\Zij$, $(i, j)\in \intbracket{1, n}\times \intbracket{1, N}$},  be independent signed Bernoulli variables, that is, \scalebox{0.88}{$\prob{\Zij=\pm 1}=\frac{1}{2}$}; and consider the random variables \scalebox{0.88}{$E_{ij}\ {\bl \in\accolade{0,1}}$} satisfying \scalebox{0.88}{$\prob{\sum_{i=1}^{n}E_{ij}=s}=1$} for some integer \scalebox{0.88}{$s\geq 1$} and all \scalebox{0.88}{$j\in \intbracket{1, N}$}. Then the matrix \scalebox{0.88}{$\Hfrac\in \rnN$} with entries \scalebox{0.88}{$\Hij:=\frac{1}{\sqrt{s}}E_{ij}\Zij$} has exactly (almost surely) $s$~nonzero elements on each column and is called an $s$-hashing matrix with hashing parameter $s$.
\end{definition}
For presentation purposes, the following definition inspired by~\cite[Definition~3]{CRsubspace2021} and~\cite[Lemma~1.1]{KaNel2014SparseLidenstrauss} is introduced.
\begin{definition}\label{JLTransforms}(Johnson--Lindenstrauss transform). 
	Let \scalebox{0.88}{$\epsilon, \delta\in (0, 1/2]$}. A random matrix \scalebox{0.88}{$\Hfrac\in \rnN$} is a JLT if, for any \scalebox{0.88}{$\x\in \rN$}, 
{\footnotesize{
\begin{equation}\label{JLTBound1}
\prob{(1-\epsilon)\normiiN{\x}\leq \normiin{\mathfrak{H}\x}\leq (1+\epsilon)\normiiN{\x}}\geq 1-\delta.
\end{equation}
}}
\end{definition}

Let \scalebox{0.88}{$\bm{e_i}, i\in \intbracket{1, n}$} be the standard basis vectors of \scalebox{0.88}{$\rn$}; and for all \scalebox{0.88}{$j\in \intbracket{1, N}$}, let \scalebox{0.88}{$\upsigma_j \subseteq\accolade{1, \dots, n}$} be independent random subsets, chosen uniformly, and with cardinality \scalebox{0.88}{$\abs{\upsigma_j}=s\leq n$}. 
%chosen according to the discrete uniform distribution. 
Let \scalebox{0.88}{$j\in \intbracket{1, N}$}, and consider the random orthogonal projection \scalebox{0.88}{$\mathcal{P}_{\upsigma_j}:=\sum_{i=1}^n \bm{e_i} \bm{e_i}^{\top}\mathds{1}_{i\in \upsigma_j}$}. Let \scalebox{0.88}{$Z_{ij}$} be independent random variables satisfying \scalebox{0.88}{$\prob{Z_{ij}=\pm 1}=\frac{1}{2}$},  and
define \scalebox{0.88}{$\mathcal{Z}_j:=(Z_{1j}, \dots, Z_{nj})^{\top}\in {\bbl \R^{n}}$} for all \scalebox{0.88}{$j$}.  Assume that the random variables \scalebox{0.88}{$Z_{ij}$} and \scalebox{0.88}{$\mathds{1}_{i\in \upsigma_j}$} are independent.
Then \scalebox{0.88}{$\mathcal{P}_{\upsigma_j} \mathcal{Z}_j = \left( Z_{1j}\mathds{1}_{1\in \upsigma_j}, \dots, Z_{nj}\mathds{1}_{n\in \upsigma_j} \right)^{\top}$}, where the random variables \scalebox{0.88}{$\mathds{1}_{i\in \upsigma_j}$} satisfy \scalebox{0.88}{$\sum_{i=1}^n\mathds{1}_{i\in \upsigma_j}=s$} almost surely. Thus, each column \scalebox{0.88}{$\mathcal{P}_{\upsigma_j} \mathcal{Z}_j$} of the random matrix \scalebox{0.88}{$\mathfrak{M}:=\left[\mathcal{P}_{\upsigma_1} \mathcal{Z}_1\ \cdots\  \mathcal{P}_{\upsigma_N} \mathcal{Z}_N   \right]\in \rnN$} has exactly~$s$ nonzero elements, and hence
\scalebox{0.88}{$\frac{1}{\sqrt{s}}\mathfrak{M}$} is an $s$-hashing matrix according to Definition~\ref{sHashingDef}, defining a JLT (see~\cite[Theorem~3.2]{DzaWildSub2022} or~\cite{KaNel2014SparseLidenstrauss}) for appropriate values of~$n$ and~$s$.  
%\vspace*{-0.25cm}
%$ $\\ 
While the random matrix \scalebox{0.88}{$\mathfrak{M}$} above has independent columns by construction, we note that entries \scalebox{0.88}{$\Yij:=Z_{ij}\mathds{1}_{i\in \upsigma_j}, i\in\intbracket{1, n}$} of a same column are not independent. \scalebox{0.88}{$\prob{\Yij=\pm 1}=\prob{\Zij=\pm 1}\prob{\mathds{1}_{i\in \upsigma_j}=1}=\frac{s}{2n}$}, while \scalebox{0.88}{$\prob{\Yij=0}=\prob{\mathds{1}_{i\in \upsigma_j}=0}=1-\frac{s}{n}$}, and \scalebox{0.88}{$\var{\Yij}=\E{\Yij^2}=\frac{s}{n}$}, since \scalebox{0.88}{$\Zij$} is independent of the Bernoulli variable \scalebox{0.88}{$\mathds{1}_{i\in \upsigma_j}$} with parameter \scalebox{0.88}{$p=\comb{n-1}{s-1}/\comb{n}{s}$}\scalebox{0.88}{$=\cfrac{(n-1)!s!(n-s)!}{(s-1)!(n-s)!n!}=\cfrac{s}{n}$}.  
\begin{proposition}\label{sProposition}
Let \scalebox{0.88}{${\bbl \tilde{\mathfrak{M}}}\in \rnN$} be a random matrix whose entries {\bbl $\tilde{Y}_{ij}$} are independent copies of a random variable~\scalebox{0.88}{$Y$} satisfying \scalebox{0.88}{$\prob{Y=\pm 1}=\frac{s}{2n}$} and \scalebox{0.88}{$\prob{Y=0}=1-\frac{s}{n}$}, for some real number \scalebox{0.88}{$s\in (0, n]$}. Then the number \scalebox{0.88}{${\bbl \mathscr{N}}$} of nonzero elements per column of \scalebox{0.88}{$\mathfrak{M}$} satisfies \scalebox{0.88}{$\E{{\bbl \mathscr{N}}}=s$}, \scalebox{0.88}{$\var{{\bbl \mathscr{N}}}=s\left(1-\frac{s}{n}\right)$}, \scalebox{0.88}{$\prob{{\bbl \mathscr{N}}=0}=\left(1-\frac{s}{n}\right)^n$}, and \scalebox{0.88}{$\prob{{\bbl \mathscr{N}}=n}=\left(\frac{s}{n}\right)^n$}.
\end{proposition}

\begin{proof}
By construction, \scalebox{0.88}{${\bbl \mathscr{N}}\sim Bin(n,\ {\bbl \tilde{p}}=\frac{s}{n})$} is a Binomial random variable with parameters \scalebox{0.88}{$n$} and \scalebox{0.88}{${\bbl \tilde{p}}=\prob{{\bbl \tilde{Y}_{ij}}\neq 0}=\frac{s}{n}$}, which immediately implies that \scalebox{0.88}{$\E{{\bbl \mathscr{N}}} = n{\bbl \tilde{p}}= s$} and \scalebox{0.88}{$\var{{\bbl \mathscr{N}}}=n{\bbl \tilde{p}}(1-{\bbl \tilde{p}})=s\left(1-\frac{s}{n}\right)$}. The other results straightforwardly follow from the distribution of \scalebox{0.88}{${\bbl \mathscr{N}}$} given by $\prob{{\bbl \mathscr{N}}=k}=${\scalebox{0.88}{$\comb{n}{k}$}}$\left(\frac{s}{n}\right)^k\left(1-\frac{s}{n}\right)^{n-k}$, {\bl for all} {\bl \scalebox{0.88}{$k=0,1,\dots,n,$}} and the proof is completed.
\end{proof}
Motivated by Proposition~\ref{sProposition}, we introduce the definition of {\bbl $s$-}hashing-like matrices {\bbl with independent entries and~$s$ expected nonzero elements on each column, as opposed to $s$-hashing matrices with exactly~$s$ nonzeros on each column and whose entries are therefore not independent.}
\begin{definition}\label{sHashingLike}($s$-hashing-like matrix).
Let \scalebox{0.88}{$\Hfrac:=\frac{1}{\sqrt{s}}{\bbl \tilde{\mathfrak{M}}}\in\rnN$}, with \scalebox{0.88}{$s\in (0, n]$} and \scalebox{0.88}{${\bbl \tilde{\mathfrak{M}}}$} defined as in Proposition~\ref{sProposition}. Then $s$ is the expected number of nonzero elements per column of \scalebox{0.88}{$\Hfrac$}, which is therefore called an $s$-hashing-like matrix. 
\end{definition}
\vspace*{-0.35cm}
$ $\\Let \scalebox{0.88}{$\mathcal{Y}_j:=(Y_{1j}, \dots, Y_{nj})^{\top}\in \R^{n\times 1}\ \forall j\in \intbracket{1,N}$}. We note from Proposition~\ref{sProposition} that if the random variables \scalebox{0.88}{$\Yij$} {\bbl defining $\mathfrak{M}$} were independent, then the number \scalebox{0.88}{${\bbl \mathscr{N}}$} of nonzero entries per column of \scalebox{0.88}{$\mathfrak{M}:=[\mathcal{Y}_1\ \cdots\ \mathcal{Y}_N]$} would satisfy \scalebox{0.88}{$\E{{\bbl \mathscr{N}}}=s$}. Motivated by this result, one might wonder whether the $s$-hashing-like matrix \scalebox{0.88}{$\mathfrak{H}:=\frac{1}{\sqrt{s}}{\bbl \tilde{\mathfrak{M}}}=\frac{1}{\sqrt{s}}(\Yij)_{\underset{1\leq j\leq N}{1\leq i\leq n}}$} with independent entries is still a JLT and, if so, for what values of $s$ and $n$. By observing on the one hand that the random variables \scalebox{0.88}{$\Xij:=\frac{1}{\sqrt{\var{\Yij}}}\Yij=\sqrt{\frac{n}{s}}\Yij$} are mean zero and unit variance and on the other hand that \scalebox{0.88}{$\mathfrak{H}$} can be rewritten as \scalebox{0.88}{$\mathfrak{H}=\frac{1}{\sqrt{n}}(\Xij)_{\underset{1\leq j\leq N}{1\leq i\leq n}}$}, the final form of the matrix ensemble analyzed in this manuscript is inspired by that in~\cite[Theorem~3.1]{MatJLT2008Variants} reported below.
\begin{theorem}\label{Matousek31}
	Let \scalebox{0.88}{$N\in \N$}, \scalebox{0.88}{$\upepsilon\in (0, 1/2]$}, \scalebox{0.88}{$\updelta\in (0,1)$}, and let \scalebox{0.88}{$n=\mathbf{C}\upepsilon^{-2}\log(\updelta/2)$}, where \scalebox{0.88}{$\mathbf{C}$} is a suitable constant. Let \scalebox{0.88}{$\mathfrak{T}:\R^N\to\rn$} defined by \scalebox{0.88}{${\mathfrak{T}(\x)}_i=\frac{1}{\sqrt{n}}\sum_{j=1}^N \Rij x_j$}, \scalebox{0.88}{$i=1, \dots, n$}, be a random linear map, where the random variables \scalebox{0.88}{$\Rij$} are independent, satisfying \scalebox{0.88}{$\E{\Rij}=0$} and \scalebox{0.88}{$\var{\Rij}=1$}, and with a uniform sub-Gaussian tail. Then for every \scalebox{0.88}{$\x\in\rN$}, \scalebox{0.88}{$\prob{(1-\upepsilon)\normiiN{\x}\leq \normiin{\mathfrak{T}(\x)} \leq(1+\upepsilon)\normiiN{\x}}\geq 1-\updelta$}.
\end{theorem}

Let \scalebox{0.88}{$q\in [2, +\infty)$} and define $\hat{q}=\sqrt{\frac{q}{2}}$ throughout the manuscript. Motivated by Theorem~\ref{Matousek31}, %and by observing that the matrix $\mathfrak{H}=\frac{1}{\sqrt{s}}\mathfrak{M}$ defined above can be rewritten as $\mathfrak{H}=\frac{1}{\sqrt{n}}(\Xij)_{\underset{1\leq j\leq N}{1\leq i\leq n}}$, 
we focus on random matrices (with independent entries) of the form
{\footnotesize{
\begin{equation}\label{HMatrixEnsemble}
\mathfrak{H}=\frac{1}{\sqrt{n}}(\Xij)_{\underset{1\leq j\leq N}{1\leq i\leq n}},\ \mbox{\normalsize with}\ \prob{\Xij=\pm \hat{q}}=\frac{1}{q}\ \mbox{\normalsize and}\ \prob{\Xij=0}=1-\frac{2}{q},\ \text{\normalsize or equivalently}
\end{equation}
}}
%or equivalently
{\footnotesize{
\begin{equation}\label{HMatrixEnsemble2}
\mathfrak{H}=\frac{\hat{q}}{\sqrt{n}}(\Yij)_{\underset{1\leq j\leq N}{1\leq i\leq n}},\ \mbox{\normalsize with}\ \prob{\Yij=\pm 1}=\frac{1}{q}\ \mbox{\normalsize and}\ \prob{\Yij=0}=1-\frac{2}{q},
\end{equation}
}}
such  that the special case $\frac{1}{q}=\frac{s}{2n}$ leads to the $s$-hashing-like matrix ensemble
{\footnotesize{
\begin{equation}\label{HashlikeEnsemble}
\mathfrak{H}=\frac{1}{\sqrt{s}}(\Yij)_{\underset{1\leq j\leq N}{1\leq i\leq n}},\ \mbox{where the}\ \Yij\ \mbox{are distributed as}\ Y\ \mbox{in Proposition~\ref{sProposition}}.
\end{equation}
}}

Let \scalebox{0.88}{$\dirac{\nu}$} denote the Dirac measure~\cite[p.~11]{tao2012topicsInRMT} centered on \scalebox{0.88}{$\nu$}.
For ease of the presentation, the following distributions are introduced:
{\footnotesize{
\begin{equation}\label{muXqmuYq}
\bbmu_X(q):=\frac{1}{q}\dirac{-\hat{q}}+\left(1-\frac{2}{q}\right)\dirac{0}+\frac{1}{q}\dirac{\hat{q}}\ \mbox{\normalsize and}\  \bbmu_Y(q):=\frac{1}{q}\dirac{-1}+\left(1-\frac{2}{q}\right)\dirac{0}+\frac{1}{q}\dirac{1}\ \mbox{\normalsize for all}\ q\geq 2,
\end{equation}
}}
so that \scalebox{0.88}{$X\sim \bbmu_X(q)$} if and only if \scalebox{0.88}{$Y:=X/\hat{q}\sim \bbmu_Y(q)$}.

\subsection{Sub-Gaussian variables and sparse Johnson--Lindenstrauss transforms}
In this section we derive sufficient conditions on {\bl \scalebox{0.88}{$q$}}, $s$ and $n$ so that {\bl the matrices given by~\eqref{HMatrixEnsemble} and~\eqref{HashlikeEnsemble}} are JLTs, using properties of sub-Gaussian variables, motivated by Theorem~\ref{Matousek31}. Indeed, even though the latter result applies trivially to the matrix ensemble{\bl s~\eqref{HMatrixEnsemble} and}~\eqref{HashlikeEnsemble} given that {\bl \scalebox{0.88}{$\Xij$} and }\scalebox{0.88}{$\sqrt{\frac{n}{s}}\Yij$} are bounded and hence sub-Gaussian~\cite[Example~2.5.8-(iii)]{vershynin2018HDP}, it does not show, e.g., how $n$ or the constant~\scalebox{0.88}{$\mathbf{C}$} depends on $s$ {\bl and \scalebox{0.88}{$q$}}. 
%In other words, Theorem~\ref{Matousek31} does not show how $s$ can be chosen so that $s$-hashing-like matrices are JLTs. 
%Thus, unlike many existing results in random matrix theory, 
Thus, the {\bl present work} particularly focuses on {\bl explicitly expressing} all {\bl non absolute} constants\footnote{\bbl By {\it absolute constants} throughout the manuscript, we mean constants that do not depend on any parameter ($\epsilon, \delta, q, s$, etc.) or variable involved in the analysis, that is, fixed real numbers that remain the same regardless of the values of such parameters.} {\bl (e.g.,  \scalebox{0.88}{$\mathbf{C}$})} in terms of $q$, $n$, and~$s$.

The remainder of this section is devoted to proving the following results.

\begin{theorem}\label{FirstJLT1}
	Let $\epsilon\in (0, 1/2]$, $\delta\in (0,1)$, {\bl and $q\geq 2$ be fixed. There exists an absolute constant $C_1>0$ such that for any $n\geq \frac{C_1 q^2}{\left(\ln\left(\frac{q}{2}+1\right)\right)^2}\epsilon^{-2}\ln(2/\delta)$, the random matrix $\mathfrak{H}\in \rnN$ given by~\eqref{HMatrixEnsemble} satisfies~\eqref{JLTBound1} for all $\x\in \R^N$.}
\end{theorem}

\begin{theorem}\label{sTheorem}
	%Let $\epsilon\in (0, 1/2]$,  $\delta\in (0, 1)$, $n\in \N$ and $0<s\leq n$. Assume that the sparsity parameter~$s$ satisfies\footnote{While some of the constants appearing in the main results are very large, the numerical experiments reported in Section~\ref{sec4} suggest significantly smaller values. The remark is similar for very small theorized constants.}
{\bl Let $\epsilon\in (0, 1/2]$ {\bl and}  $\delta\in (0, 1)$. There exists an absolute constant $C_1'>0$ such that if $s$ satisfies,
{\footnotesize{
\begin{equation}\label{sCondition}
s^2\geq {\bl C_1'}\epsilon^{-2}\frac{n\ln(2/\delta)}{\left(\ln\left(\frac{n}{s}+1\right)\right)^2},
\end{equation}
}}
then the $s$-hashing-like matrix $\mathfrak{H}\in \rnN$ of~\eqref{HashlikeEnsemble} satisfies~\eqref{JLTBound1} for all $\x\in \R^N$.
}
\end{theorem}

Since $s$ is naturally at most equal to $n$, how~\eqref{sCondition} can be satisfied is not obvious. A demonstration of this satisfiability is therefore provided in the following remark.

\begin{remark}
	Choose \scalebox{0.88}{$\epsilon=\delta=\frac{1}{2}$} in Theorem~\ref{sTheorem}, and assume that \scalebox{0.88}{$s=\upkappa n$} for some \scalebox{0.88}{$\upkappa\in (0, 1]$}, where \scalebox{0.88}{$n\geq\uptau:=\frac{{\bl 4C_1'}\ln 4}{(\ln 2)^2}$}.
	Observe that \scalebox{0.88}{$\ln\left(\frac{n}{s}+1\right)\geq \ln 2$}. Then for~\eqref{sCondition} to hold, it is sufficient that \scalebox{0.88}{$s^2=\upkappa^2 n^2\geq \uptau n$}, which is the case provided \scalebox{0.88}{$1\geq \upkappa\geq \sqrt{\frac{\uptau}{n}}\in (0, 1]$}.
	
	More generally, define \scalebox{0.88}{$\uptau(\epsilon, \delta):={\bl C_1'}\epsilon^{-2}\frac{\ln(2/\delta)}{(\ln 2)^2}$} for fixed \scalebox{0.88}{$\epsilon\in (0,1/2]$} and \scalebox{0.88}{$\delta\in (0,1)$}, and choose \scalebox{0.88}{$n\geq \uptau(\epsilon, \delta)$}. Then as above, for~\eqref{sCondition} to hold with \scalebox{0.88}{$s=\upkappa(\epsilon, \delta) n$}, it is sufficient that \scalebox{0.88}{$1\geq \upkappa(\epsilon, \delta)\geq \sqrt{\frac{\uptau(\epsilon, \delta)}{n}}=\frac{{\bl 1}}{\epsilon\ln 2}\sqrt{\frac{{\bl C_1'}\ln (2/\delta)}{n}}\in (0,1]$}.
\end{remark}

The proof of Theorem~\ref{FirstJLT1} follows that of Theorem~\ref{Matousek31}, that is, ~\cite[Theorem~3.1]{MatJLT2008Variants}, and therefore requires the key result of Lemma~\ref{expoLambdaU} below. The main idea is to focus on the expression of \scalebox{0.88}{$\Hfrac$} given by~\eqref{HMatrixEnsemble}, in order to avoid that the distribution of the \scalebox{0.88}{$\Xij$} depends on both  $n$ and $s$, but only on a fixed \scalebox{0.88}{$q\geq 2$}, thus simplifying the analysis, leading to Theorem~\ref{FirstJLT1}. Then the result of Theorem~\ref{sTheorem} will follow as a corollary of Theorem~\ref{FirstJLT1} in the particular case where \scalebox{0.88}{$q = \frac{2n}{s}$}.  Since~\eqref{HMatrixEnsemble},~\eqref{HMatrixEnsemble2} and~\eqref{HashlikeEnsemble} can all be expressed in terms of \scalebox{0.88}{$\Yij$}, the analysis below {\bl therefore} focuses on properties of \scalebox{0.88}{$\Yij$}, the most basic being given by Proposition~\ref{Yproporties}.

Next is introduced the important class of sub-Gaussian random variables with strong tail decay properties~\cite[Section~3]{Rudelson2014LectNotes}: a class containing bounded variables as well as normal variables.
\begin{definition}
Let \scalebox{0.88}{$\mathfrak{v}>0$}. A random variable \scalebox{0.88}{$S$} is called \scalebox{0.88}{$\mathfrak{v}$}-sub-Gaussian if there exists a constant \scalebox{0.88}{$\mathfrak{C}$} such that for any \scalebox{0.88}{$t\geq 0$}, \scalebox{0.88}{$\prob{\abs{S}>t}\leq \mathfrak{C}e^{-\mathfrak{v} t^2}$}. If \scalebox{0.88}{$\mathfrak{v}$} is an absolute constant, then a \scalebox{0.88}{$\mathfrak{v}$}-sub-Gaussian variable is simply called sub-Gaussian.
\end{definition}

For any sub-Gaussian random variable \scalebox{0.88}{$S$}, the following {\it sub-Gaussian norm}~\cite[Definition~2.5.6]{vershynin2018HDP} will be considered throughout: \scalebox{0.88}{$\normPsiTwo{S}=\inf\accolade{t>0:\E{\exp\left(S^2/t^2\right)}\leq 2}.$}
The following properties of \scalebox{0.88}{$Y\sim\bbmu_Y(q)$} defined by \eqref{muXqmuYq}, will be useful to derive the results of Propositions~\ref{HoeffdingTypeIneq} and~\ref{KhintchineProposition}, which will be needed to prove Lemma~\ref{expoLambdaU}.
\begin{proposition}\label{Yproporties}
The random variable \scalebox{0.88}{$Y\sim \bbmu_Y(q)$} defined by \eqref{muXqmuYq} satisfies the following properties for all \scalebox{0.88}{$q\geq 2$}, {\bl with absolute constants \scalebox{0.88}{$C_2>1$} and \scalebox{0.88}{$C_3>1$} (independent of $q$)}:
\begin{itemize}
\item[(i)] \scalebox{0.88}{$\normPsiTwo{Y}=:K_1=\frac{1}{\sqrt{\ln\left(\frac{q}{2}+1\right)}}$}.
\item[(ii)] \scalebox{0.88}{$\prob{\abs{Y}\geq t}\leq 2\exp\left(-t^2/K_1^2\right)\ $} for all \scalebox{0.88}{$t\geq 0$}.
\item[(iii)] \scalebox{0.88}{$\left(\E{\abs{Y}^p}\right)^{1/p}\leq K_2\sqrt{p}\ $} for all \scalebox{0.88}{$p\geq 1$}, where \scalebox{0.88}{$K_2 = {\bl C_2K_1}$}.
\item[(iv)] \scalebox{0.88}{$\E{e^{\lambda^2 Y^2}}\leq e^{K_3^2 \lambda^2}\ $} \scalebox{0.88}{$\forall\abs{\lambda}\leq\frac{1}{K_3}$},  and \scalebox{0.88}{$\ \E{e^{\lambda Y}}\leq e^{K_3^2 \lambda^2}\ $} \scalebox{0.88}{$\forall\lambda\in \R$}, where \scalebox{0.88}{$K_3 = {\bl C_3K_2}$}.
\end{itemize}
\end{proposition}
\begin{proof}
To prove~{\it (i)}, we first observe that the equality $\normPsiTwo{Y}=\frac{1}{\sqrt{\ln\left(\frac{q}{2}+1\right)}}$ follows from the fact that $\E{\exp\left(Y^2/t^2\right)}-2 = \frac{2}{q}e^{1/t^2}+\left(1-\frac{2}{q}\right)-2\leq 0\ \iff\ t^2\geq \frac{1}{\ln\left(\frac{q}{2}+1\right)}.$
	
The proof of~{\it (ii)} uses the fact that $Y$ is sub-Gaussian and hence satisfies~(see, e.g.,~\cite[Eq.~(2.14)]{vershynin2018HDP}) for some constant $c_1>0$, 
{\footnotesize{
\begin{equation}\label{boundSubGauss}
\prob{\abs{Y}\geq t}\leq 2\exp\left(-c_1t^2/\normPsiTwo{Y}^2\right)=2\left(\frac{q}{2}+1\right)^{-c_1t^2} \ \mbox{\normalsize for all}\ t\geq 0.
\end{equation}
}}Simple calculations using \scalebox{0.88}{$\bbmu_Y(q)$} yield \scalebox{0.88}{$\prob{\abs{Y}\geq t}=\mathds{1}_{\accolade{t\, =\, 0}} + \frac{2}{q}\mathds{1}_{\accolade{0<t\leq 1}}\leq 2\left(\frac{q}{2}+1\right)^{-t^2}$}
for all \scalebox{0.88}{$t\geq 0$}, which shows that~\eqref{boundSubGauss} holds for \scalebox{0.88}{$c_1=1$}, leading to~{\it (ii)}.
	
The proofs of the other results use ideas derived from~\cite[Proof of Proposition~2.5.2]{vershynin2018HDP}. To prove~{\it (iii)}, let \scalebox{0.88}{$Z:=Y/K_1$}. It follows from~{\it (ii)} that \scalebox{0.88}{$\prob{\abs{Z}\geq t}\leq 2e^{-t^2}$}. By the integral identity~\cite[Lemma~1.2.1]{vershynin2018HDP} with the change of variable $u=t^p$, and the inequality \scalebox{0.88}{$\Gamma(x)\leq 3x^x$} for all \scalebox{0.88}{$x\geq 1/2$}, where \scalebox{0.88}{$\Gamma$} denotes the {\it Gamma function},
{\footnotesize{
\begin{eqnarray*}
\E{\abs{Z}^p}&=&\int_{0}^{\infty}\prob{\abs{Z}^p\geq u}du=\int_{0}^{\infty}\prob{\abs{Z}\geq t}pt^{p-1}dt\\
&\leq&\int_{0}^{\infty}2e^{-t^2}pt^{p-1}dt=p\Gamma(p/2)\leq 3p(p/2)^{p/2},
\end{eqnarray*}
}}which implies \scalebox{0.88}{$\left(\E{\abs{Z}^p}\right)^{1/p}\leq (3p)^{1/p}\cdot\sqrt{\cfrac{p}{2}}\leq \cfrac{3}{\sqrt{2}}\sqrt{p}$}, and hence {\bl \scalebox{0.88}{$\left(\E{\abs{Y}^p}\right)^{1/p}\leq \cfrac{3 K_1}{\sqrt{2}}\sqrt{p}$}}.

To prove the first inequality in~{\it (iv)}, we observe that~{\it (iii)} applied to \scalebox{0.88}{$Y=:K_2W$} leads to \scalebox{0.88}{$\E{W^{2p}}\leq (2p)^p$}, which allows us to derive in the proof of~\cite[Proposition~2.5.2]{vershynin2018HDP} that \scalebox{0.88}{$\E{e^{\lambda^2W^2}}=1+\sum_{p=1}^{\infty}\frac{\lambda^{2p}\E{W^{2p}}}{p!}\leq e^{4e\lambda^2}$}, \scalebox{0.88}{$\forall\abs{\lambda}\leq \frac{1}{2\sqrt{e}}.$} Hence, \scalebox{0.88}{$\E{e^{\lambda^2Y^2}}\leq e^{4eK_2^2\lambda^2}=e^{K_3^2\lambda^2}$} for all \scalebox{0.88}{$\abs{\lambda}{\bl \leq 1/2K_2\sqrt{e}}$}.

By applying the first inequality in~{\it (iv)} to \scalebox{0.88}{$Y=:K_3U$}, it holds that \scalebox{0.88}{$\E{e^{\lambda^2U^2}}\leq e^{\lambda^2}$} for all \scalebox{0.88}{$\abs{\lambda}\leq 1$}. Then the inequality \scalebox{0.88}{$e^x\leq x+e^{x^2}$}, together with the fact that \scalebox{0.88}{$\E{U}=0$}, yields \scalebox{0.88}{$\E{e^{\lambda U}}\leq e^{\lambda^2}$} for all \scalebox{0.88}{$\abs{\lambda}\leq 1$}. When \scalebox{0.88}{$\abs{\lambda}\geq 1$},  it follows from the inequalities \scalebox{0.88}{$\lambda x\leq \frac{\lambda^2}{2}+\frac{x^2}{2}$} for all \scalebox{0.88}{$\lambda$} and \scalebox{0.88}{$x$} and \scalebox{0.88}{$\E{e^{U^2/2}}\leq e^{1/2}$} that \scalebox{0.88}{$\E{e^{\lambda U}}\leq e^{\lambda^2/2}\E{e^{U^2/2}}\leq e^{\lambda^2/2}e^{1/2}\leq e^{\lambda^2}$}. Consequently, \scalebox{0.88}{$\E{e^{\lambda U}}\leq e^{\lambda^2}$} and hence \scalebox{0.88}{$\E{e^{\lambda Y}}\leq e^{K_3^2 \lambda^2}$}, for all \scalebox{0.88}{$\lambda\in\R$}.
\end{proof}

The following result is a particular case of the more general large deviation inequality~\cite[Theorem~3.3]{Rudelson2014LectNotes} (see also~\cite[Theorem~2.6.3]{vershynin2018HDP}) for linear combinations of independent sub-Gaussian random variables that states that a linear combination of independent sub-Gaussian random variables is sub-Gaussian. It will be used not only to prove Proposition~\ref{KhintchineProposition} but also to estimate the largest singular value of hashing-like matrices via Theorem~\ref{LargestSingResult}.
\begin{proposition}\label{HoeffdingTypeIneq}(Hoeffding-type inequality).
Let \scalebox{0.88}{$Y_{\ell},\ \ell\in \intbracket{1,m}$}, be independent copies of \scalebox{0.88}{$\ Y\sim \bbmu_Y(q)$}. Consider the deterministic vector \scalebox{0.88}{$\ab:=(a_1, \dots, a_m)^{\top}\in \mathbb{S}^{m-1}$}. Then {\bl there exists an absolute constant \scalebox{0.88}{$c_{2,3}\in (0,1)$} such that}
{\footnotesize{
\begin{equation}\label{HoeffIneq1}
\prob{\abs{\sum_{\ell=1}^{m}a_\ell Y_\ell}\geq t}\leq 2e^{-c_2t^2}\quad\text{\normalsize for all}\ t\geq 0, \ \mbox{\normalsize where}\ {\bl c_2=\left(\frac{c_{2,3}}{K_1}\right)^2}.
\end{equation}
}}
\end{proposition}

\begin{proof}
Consider \scalebox{0.88}{$S_{Y,m}:= \sum_{\ell=1}^{m}a_\ell Y_\ell$}. By the independence of the \scalebox{0.88}{$Y_\ell$}, for all \scalebox{0.88}{$\lambda\in\R$},
{\footnotesize{
\begin{equation}\label{expoSyn}
\E{e^{\lambda S_{Y,m}}}=\prod_{\ell=1}^{m}\E{e^{\lambda a_\ell Y_\ell}}\leq \prod_{\ell=1}^{m} e^{\lambda^2 K_3^2 a_\ell^2}= e^{\lambda^2 K_3^2 (\sum_{\ell=1}^m a_\ell^2)}= e^{\lambda^2 K_3^2},
\end{equation}
}}where the inequality is due to the second result in~Proposition~\ref{Yproporties}-{\it (iv)}, while the last equality follows from the fact that \scalebox{0.88}{$\normii{\bm{a}}=1$}.
By observing that \scalebox{0.88}{$Y_\ell$}, and consequently \scalebox{0.88}{$S_{Y,m}$} are symmetric random variables, it follows from the Markov inequality that for all \scalebox{0.88}{$\lambda,\ t\geq 0$},
\scalebox{0.88}{$\prob{\abs{S_{Y,m}}\geq t}=\prob{S_{Y,m}\geq t}+\prob{-S_{Y,m}\geq t}=2\prob{S_{Y,m}\geq t}\leq 2\frac{\E{e^{\lambda S_{Y,m}}}}{e^{\lambda t}},$}
which,\\ combined with~\eqref{expoSyn}, yield \scalebox{0.88}{$\prob{\abs{S_{Y,m}}\geq t}\leq2 e^{\lambda^2 K_3^2-\lambda t}$}  \scalebox{0.88}{$\forall\lambda\geq 0$} whence \scalebox{0.88}{$\prob{\abs{S_{Y,m}}\geq t}\leq$} \scalebox{0.88}{$2e^{\inf_{\lambda\geq 0}\left(\lambda^2 K_3^2-\lambda t\right)}=2e^{-\frac{t^2}{4K_3^2}} {\bl = 2e^{-c_2t^2},}$} where \scalebox{0.88}{$c_2:= \frac{1}{(2C_2C_3K_1)^2}$} thanks to Proposition~\ref{Yproporties}.
\end{proof}
An application of the Hoeffding-type inequality above leads to the following result providing an upper bound on the \scalebox{0.88}{$\mathbb{L}^p$}-norm of the linear combination \scalebox{0.88}{$\sum_{\ell=1}^{m}a_\ell Y_\ell$} of Proposition~\ref{HoeffdingTypeIneq}, which is one of the inequalities in the more general result known as ``Khintchine's inequality''~\cite[Theorem~3.4]{Rudelson2014LectNotes} and~\cite[Exercises~2.6.5 and~2.6.6]{vershynin2018HDP}.
\begin{proposition}\label{KhintchineProposition}(Khintchine's upper bound). 
Consider \scalebox{0.88}{$Y_{\ell}\sim \bbmu_Y(q),\ \ell\in \intbracket{1,m}$} (recalling that \scalebox{0.88}{$q\geq 2$} throughout) and the deterministic vector \scalebox{0.88}{$\ab:=(a_1, \dots, a_m)^{\top}\in \mathbb{S}^{m-1}$} as in Proposition~\ref{HoeffdingTypeIneq}. Then, {\bl there exists an absolute constant \scalebox{0.88}{$C_4>1$} such that}
{\footnotesize{
\begin{equation}\label{Khintchine}
\left(\E{\abs{\sum_{\ell=1}^{m}a_\ell Y_\ell}^p}\right)^{1/p}\leq K_4\sqrt{p}\quad \mbox{\normalsize for all}\ p\geq 1,\ \mbox{\normalsize where}\ {\bl K_4=\frac{C_4}{\sqrt{\ln\left(\frac{q}{2}+1\right)}}}.
\end{equation}
}}
\end{proposition}
\begin{proof}
By Proposition~\ref{HoeffdingTypeIneq}, \scalebox{0.88}{$S_{Y,m}:= \sum_{\ell=1}^{m}a_\ell Y_\ell$} satisfies \scalebox{0.88}{$\prob{\abs{\sqrt{c_2}S_{Y,m}}\geq t}\leq 2e^{-t^2}$} \scalebox{0.88}{$\forall t\geq 0$}. Thus, following the proof of Proposition~\ref{Yproporties}-{\it (iii)}, \scalebox{0.88}{$\left(\E{\abs{\sqrt{c_2}S_{Y,m}}^p}\right)^{1/p}\leq \frac{3}{\sqrt{2}}\sqrt{p}$} for all \scalebox{0.88}{$p\geq 1$}, {\bl leading to \scalebox{0.88}{$\left(\E{\abs{S_{Y,m}}^p}\right)^{1/p}\leq \frac{3}{\sqrt{2c_2}}\sqrt{p}=\frac{3}{c_{2,3}\sqrt{2}}K_1\sqrt{p}$} for all \scalebox{0.88}{$p\geq 1$}, which completes the proof}.
\end{proof}

The following result in the spirit of~\cite[Lemma~3.3]{MatJLT2008Variants} provides an inequality related to the {\it moment generating function} \scalebox{0.88}{$\E{e^{\lambda U}}$} of a centered subexponential random variable \scalebox{0.88}{$U$} (see, e.g.,~\cite[Proposition~2.7.1-{\it (e)}]{vershynin2018HDP}). 
%Unlike~\cite[Lemma~3.3]{MatJLT2008Variants}. 
%it is stated and proved by using known and explicit constants.
\begin{lemma}\label{expoLambdaU}
Let \scalebox{0.88}{$\x\in\sNOne$} be fixed, and define \scalebox{0.88}{$T_i=\sum_{j=1}^N x_j X_{ij}$}, $i=1, \dots, n$, where the $\Xij\sim \bbmu_X(q)$ are independent, and recall that $q\geq 2$ throughout.~Consider \scalebox{0.88}{$U:=\frac{1}{\sqrt{n}}\left(T_1^2+T_2^2+\cdots+T_n^2-n\right)$}. {\bl There exists an absolute constant \scalebox{0.88}{$C_5>0$} such that}
{\footnotesize{
\begin{equation}\label{expoUineq}
\E{e^{\lambda U}}\leq e^{K_5^2\lambda^2/2} \quad\text{\normalsize for all}\  \abs{\lambda}\leq \frac{1}{K_5}\sqrt{n}, \ \text{\normalsize where}\ {\bl K_5 = \frac{C_5 q}{\ln\left(\frac{q}{2}+1\right)}}.
\end{equation}
}}
\end{lemma}

\begin{proof}
For \scalebox{0.88}{$\x\in\sNOne$} fixed, the random variables $T_i$ are independent and identically distributed, which yields
{\footnotesize{
\begin{equation}\label{firstLambdaUineq}
\E{e^{\lambda U}}=\E{e^{(\lambda/\sqrt{n}) \sum \limits_{i=1}^n(T_i^2-1)}}=\left(\E{e^{(\lambda/\sqrt{n}) (T_1^2-1)}}\right)^n.
\end{equation}
}}
Then the remainder of the proof involves three steps. The first one establishes a Bernstein-like condition, which serves to derive a bound on \scalebox{0.88}{$\E{e^{\lambda(T^2-1)}}$} by means of a Taylor expansion in step~2, for appropriate values of $\lambda$. The third step completes the proof by putting all ingredients together. \\ 
{\it Step~1: ({\em Bernstein}-like condition)}. To upper bound the last term in~\eqref{firstLambdaUineq}, we observe, using the equality \scalebox{0.88}{$\Xij=\hat{q}\Yij$} and Khintchine's upper bound~\eqref{Khintchine}, that for all $p\geq 1$,
{\footnotesize{
\begin{equation}\label{T1Khintcine}
\left(\E{\abs{T_1}^p} \right)^{1/p}=\left(\E{\abs{\sum_{j=1}^N x_j X_{1j}}^p} \right)^{1/p}=\hat{q}\left(\E{\abs{\sum_{j=1}^N x_j Y_{1j}}^p} \right)^{1/p}\leq K_4\sqrt{\frac{pq}{2}}, 
\end{equation}
}}
which implies the Bernstein-like condition \scalebox{0.88}{$\E{\abs{T_1^2}^k}\leq K_4^{2k}(kq)^k$}, for all \scalebox{0.88}{$k\geq 1$}. Consider \scalebox{0.88}{$\upzeta:=T_1^2-1$}, which is centered since \scalebox{0.88}{$\E{T_1^2}=\E{\sum_{j=1}^N x_j^2 X_{1j}^2}=\E{X_{11}^2}\sum_{j=1}^N x_j^2=1$} given that the \scalebox{0.88}{$\Xij$} are i.i.d., mean zero, and \scalebox{0.88}{$\x\in \sNOne$}. Using  \scalebox{0.88}{$(\upnu+\upmu)^k\leq 2^{k-1}(\upnu^k+\upmu^k)$} for all \scalebox{0.88}{$k\geq 1$} and \scalebox{0.88}{$\upnu, \upmu\geq 0$} and the fact that \scalebox{0.88}{$K_4^2 q \geq 1$} for all \scalebox{0.88}{$q\geq 2$}, we have that for all \scalebox{0.88}{$k\geq 1$},
{\footnotesize{
\begin{equation}\label{zetaCondition}
\E{\abs{\upzeta}^k}\leq \E{\left(\abs{T_1^2}+\!1\right)^k}\leq 2^{k-1}\left(\E{\abs{T_1^2}^k}+\!1 \right)\leq 2^{k-1} \left(K_4^{2k}(kq)^k+\!1\right)\leq 2^k(kq)^k K_4^{2k}.
\end{equation}
}}
{\it Step~2:  (Taylor expansion)}.  Recalling that \scalebox{0.88}{$\E{\upzeta}=0$} and using the inequality \scalebox{0.88}{$k!\geq k^k e^{-k}$} for all \scalebox{0.88}{$k\geq 2$}, which follows from the Stirling formula (see, e.g.,~\cite[Section~1.2]{tao2012topicsInRMT}), we have that for all \scalebox{0.88}{$\abs{\lambda}<1/K_5'$} with \scalebox{0.88}{$K_5':=2qeK_4^2$},
{\footnotesize{
\begin{eqnarray*}
\E{e^{\lambda\upzeta}}&=&\E{1+\lambda\upzeta+\sum_{k=2}^{\infty}\frac{\lambda^k\upzeta^k}{k!} }\leq 1+\sum_{k=2}^{\infty}\frac{\abs{\lambda}^k}{k!}\E{\abs{\upzeta}^k} \leq 1+\sum_{k=2}^{\infty} \abs{\lambda}^k(2qK_4^2)^k\frac{k^k}{k!} \\
&\leq& 1+\sum_{k=2}^{\infty}(K_5'\abs{\lambda})^k=1+\frac{(K_5'\lambda)^2}{1-K_5'\abs{\lambda}}\leq e^{\frac{(K_5'\lambda)^2}{1-K_5'\abs{\lambda}}},
\end{eqnarray*}
}}
where the expectation and the infinite sum are interchanged via the first inequality using Fubini--Tonelli's theorem (see e.g.,~\cite[Theorem~1.7.2]{durrett2019probability}) since \scalebox{0.88}{$\abs{\lambda}^k\abs{\upzeta}^k/k!\geq 0$}. The second inequality follows from~\eqref{zetaCondition}, and the last one results from \scalebox{0.88}{$1+\mathfrak{t}\leq e^{\mathfrak{t}}$} for all $\mathfrak{t}\in\R$. Let \scalebox{0.88}{$K_5 = 2K_5' {\bl = 4qeK_4^2=\frac{4eC_4^2q}{\ln\left(\frac{q}{2}+1\right)}}$}. Then the condition \scalebox{0.88}{$\abs{\lambda}\leq \frac{1}{K_5}$} implies \scalebox{0.88}{$1-K_5'\abs{\lambda}\geq \frac{1}{2}$}. Hence, \scalebox{0.88}{$\E{e^{\lambda\upzeta}}\leq e^{\frac{K_5^2\lambda^2}{2}}$}, for all \scalebox{0.88}{$\abs{\lambda}\leq \frac{1}{K_5}$}.\\
{\it Step 3: (Putting all ingredients together)}. It follows from both latter inequalities and~\eqref{firstLambdaUineq} that\\ \scalebox{0.88}{$\E{e^{\lambda U}}\leq \left(\E{e^{\frac{\lambda}{\sqrt{n}}\upzeta}}\right)^n\leq \left(e^{\frac{K_5^2\lambda^2}{2n}}\right)^n=e^{\frac{K_5^2\lambda^2}{2}}, \text{\normalsize for all}\ \abs{\frac{\lambda}{\sqrt{n}}}\leq \frac{1}{K_5},$} which completes the proof.
\end{proof}

The next result in the spirit of~\cite[Proposition~3.2]{MatJLT2008Variants} states that the random variable \scalebox{0.88}{$U$} of Lemma~\ref{expoLambdaU} has a (see, e.g.,~\cite[Definition~2.1]{MatJLT2008Variants}) {\it sub-Gaussian tail up to} \scalebox{0.88}{$2\tilde{K}_5\sqrt{n}$}, for any \scalebox{0.88}{$\tilde{K}_5\in (0, K_5)$}, and is essential for the proof of Theorem~\ref{FirstJLT1}.
\begin{lemma}\label{Ukt}
Let \scalebox{0.88}{$U$} be defined as in Lemma~\ref{expoLambdaU}. Then for all \scalebox{0.88}{$t\in (0, 2K_5\sqrt{n})$}, it holds that \scalebox{0.88}{$\prob{U\geq t}\leq e^{-\frac{1}{4K_5^2}t^2}$} and \scalebox{0.88}{$\prob{-U\geq t}\leq e^{-\frac{1}{4K_5^2}t^2}.$}
\end{lemma}

\begin{proof}
To prove the first inequality, fix \scalebox{0.88}{$t\in (0, 2K_5\sqrt{n})$} and consider the function \scalebox{0.88}{$\wp_1(\lambda):=-\lambda t+K_5^2\lambda^2$} defined for all \scalebox{0.88}{$\lambda\in \mathbb{I}_1:=\left(0, K_5^{-1}\sqrt{n}\right)$}. Then  \scalebox{0.88}{$\lambda^{\star}_1:=\argmin\accolade{\wp_1(\lambda):\lambda\in \mathbb{I}_1}=\frac{t}{2K_5^2}$} satisfies \scalebox{0.88}{$\wp_1(\lambda^{\star}_1)=-\frac{t^2}{4K_5^2}$}. Thus, by Markov's inequality, it holds that for all \scalebox{0.88}{$\lambda\in \mathbb{I}_1$},
\scalebox{0.88}{$\prob{U\geq t} = \prob{e^{\lambda U}\geq e^{\lambda t}}\leq e^{-\lambda t}\E{e^{\lambda U}}\leq  e^{\wp_1(\lambda)}, $}
the last inequality following from~\eqref{expoUineq}, which implies that $\prob{U\geq t} \leq e^{\min \accolade{\wp_1(\lambda):\ \lambda\in \mathbb{I}_1}}=e^{-\frac{1}{4K_5^2}t^2}$.
	
To derive the second inequality, we use the same previous arguments with the function \scalebox{0.88}{$\wp_2(\lambda):=\lambda t+K_5^2\lambda^2$} for all \scalebox{0.88}{$\lambda\in \mathbb{I}_2:=\left(-K_5^{-1}\sqrt{n}, 0\right)$}, where \scalebox{0.88}{$t\in (0, 2K_5\sqrt{n})$} is fixed, and then observe that \scalebox{0.86}{$\lambda^{\star}_2:=\argmin\accolade{\wp_2(\lambda):\lambda\in \mathbb{I}_2}=-\frac{t}{2K_5^2}$} satisfies \scalebox{0.86}{$\wp_2(\lambda^{\star}_2)=-\frac{t^2}{4K_5^2}$}. Thus, for all \scalebox{0.88}{$\lambda\in \mathbb{I}_2$},
\scalebox{0.88}{$\prob{-U\geq t} = \prob{e^{\lambda U}\geq e^{-\lambda t}}\leq e^{\lambda t}\E{e^{\lambda U}}\leq  e^{\wp_2(\lambda)}, $} which finally implies that \scalebox{0.88}{$\prob{-U\geq t}  \leq e^{\min \accolade{\wp_2(\lambda):\ \lambda\in \mathbb{I}_2}}=e^{-\frac{1}{4K_5^2}t^2}$},  and the proof is completed.
\end{proof}
The result of Theorem~\ref{FirstJLT1} from which is derived  Theorem~\ref{sTheorem} is in the spirit of~\cite[Theorem~4.1]{MatJLT2008Variants} (see also~\cite[Theorem~2]{FedSchJohHeo2018}) reported below  for comparison purposes.
\begin{theorem}\label{MatTheorem42}
Let \scalebox{0.88}{$\epsilon\in (0, 1/2]$, $\delta\in (0,1)$} and \scalebox{0.88}{$\upalpha\in \left[\frac{1}{\sqrt{N}}, 1\right]$}. Consider the matrix \scalebox{0.88}{$\Hfrac\in \rnN$} given by~\eqref{HMatrixEnsemble}, defined by \scalebox{0.88}{$\Xij~\sim \bbmu_X(q)$}, with \scalebox{0.88}{$q=2/C_0\upalpha^2\ln(N/\epsilon \delta)\geq 2$} for some sufficiently large constant \scalebox{0.88}{$C_0$}. Let \scalebox{0.88}{$n=C\epsilon^{-2}\ln(4/\delta)$} for a sufficiently large constant $C$ (assuming $n$ integral). Then for any \scalebox{0.88}{$\x\in\sNOne$} satisfying \scalebox{0.88}{$\norminf{\x}\leq \upalpha$, $\mathfrak{H}$} satisfies~\eqref{JLTBound1}.
\end{theorem}
As mentioned in the first paragraph of~\cite[Section~4]{MatJLT2008Variants}, Theorem~\ref{MatTheorem42} {\bl only holds} for~$\x$ with sufficiently small \scalebox{0.88}{$\ell_{\infty}$}-norm. Moreover, {\bl it holds} for $q$ depending on \scalebox{0.88}{$N$}, which is not the case for Theorem~\ref{FirstJLT1} {\bl proved} below, providing a sufficient condition on $n$ so that the matrix \scalebox{0.88}{$\Hfrac$} given by~\eqref{HMatrixEnsemble} is a JLT. 
\subsubsection*{Proof of Theorem~\ref{FirstJLT1}}
\begin{proof}
The proof basically follows that of~\cite[Theorem~3.1]{MatJLT2008Variants}. Without any loss of generality, we assume \scalebox{0.88}{$\x\in \sNOne$}. Let the vector \scalebox{0.88}{$\x$} be fixed, and consider the same random variables \scalebox{0.88}{$T_i, i\in\intbracket{1,n}$} and \scalebox{0.88}{$U$} of Lemma~\ref{expoLambdaU}. 
{\bl Recall \scalebox{0.88}{$K_5 := \frac{C_5 q}{\ln\left(\frac{q}{2}+1\right)}$} from Lemma~\ref{expoLambdaU}, where \scalebox{0.88}{$C_5>0$} is an absolute constant.} We first observe that the $i$th component of \scalebox{0.88}{$\Hfrac\x$} is given by \scalebox{0.88}{$(\Hfrac\x)_i=\frac{1}{\sqrt{n}} T_i$}, which implies \scalebox{0.88}{$\normiin{\mathfrak{H}\x}^2-1=\frac{1}{\sqrt{n}} U$}. Thus, 
\scalebox{0.88}{$\prob{\normiin{\mathfrak{H}\x}\geq 1+\epsilon}\leq \prob{\normiin{\mathfrak{H}\x}^2\geq 1+2\epsilon}=\prob{U\geq 2\epsilon\sqrt{n}}$} \scalebox{0.88}{$\leq e^{-\frac{1}{4K_5^2}(2\epsilon\sqrt{n})^2}\leq \frac{\delta}{2}$},
where the second inequality follows from the first result of Lemma~\ref{Ukt} since \scalebox{0.88}{$2\epsilon\sqrt{n} \in (0, 2K_5\sqrt{n})$},  while the last one follows from the choice of~$n$ {\bl according to \scalebox{0.88}{$n\geq K_5^2\epsilon^{-2}\ln(2/\delta)$}}.  Similarly, it holds that\\
\scalebox{0.88}{$\prob{\normiin{\mathfrak{H}\x}\leq 1-\epsilon}\leq \prob{\normiin{\mathfrak{H}\x}^2\leq 1-2\epsilon}=\prob{-U\geq 2\epsilon\sqrt{n}}\leq e^{-\frac{1}{4K_5^2}(2\epsilon\sqrt{n})^2}\leq \frac{\delta}{2}$} by the second result of Lemma~\ref{Ukt}. Then noticing that 
\scalebox{0.86}{$\prob{(1-\epsilon)\leq \normiin{\mathfrak{H}\x}\leq (1+\epsilon)} \geq 1-\prob{\normiin{\mathfrak{H}\x}\leq 1-\epsilon}-\prob{\normiin{\mathfrak{H}\x}\geq 1+\epsilon}$}, completes the proof.
\end{proof}

The proof of Theorem~\ref{sTheorem} (showing how $n$ and $s$ can be chosen so that \scalebox{0.88}{$\Hfrac$} given by~\eqref{HashlikeEnsemble} is a JLT) is presented below.
\subsubsection*{Proof of Theorem~\ref{sTheorem}}
\begin{proof}
The result straightforwardly follows from Theorem~\ref{FirstJLT1} since the inequality {\bl \scalebox{0.88}{$n\geq \frac{C_1 q^2}{\left(\ln\left(\frac{q}{2}+1\right)\right)^2}\epsilon^{-2}\ln(2/\delta)$}} implies \scalebox{0.88}{$s^2\geq 4{\bl C_1}\epsilon^{-2}\frac{n\ln(2/\delta)}{\left(\ln\left(\frac{n}{s}+1\right)\right)^2}$} when \scalebox{0.88}{$q=\frac{2n}{s}$}.
\end{proof}
{\bbl
\begin{remark}
We conclude this section by emphasizing that the goal of this manuscript is not only to propose a class of sparse JLTs but rather to also have ready a framework allowing a variant of the JL lemma with no potential restriction to the unit sphere that can, at the same time, allow and easy the analysis of extreme singular values as presented in Section~\ref{sec3}. Indeed, while the result of Kane and Nelson~\cite{KaNel2014SparseLidenstrauss} on $s$-hashing matrices (with dependent entries) is stronger than ours, there does not exist, to our knowledge, any work which is specific to the singular values of such matrices. This requires us to introduce independence in the entries, leading to $s$-hashing-like matrices, which not only worsens the result in~\cite{KaNel2014SparseLidenstrauss}, but is also the price to pay to analyze the singular values of the latter matrices. On the other hand, it is also worth recalling the fast Johnson--Lindenstrauss transform (FJLT, based on hashing techniques) of Ailon and Chazelle~\cite{AilChaz2006} as a very strong competitor of the sparse JLT proposed in this manuscript. However, we are not aware of any work on the analysis of the singular values of the FJLT.
\end{remark}

}

\section{Extreme singular values: asymptotic and non-asymptotic behaviors}\label{sec3}
In this section we investigate both the asymptotic and non-asymptotic behaviors of the extreme singular values of the sparse matrices (with independent entries) introduced in Section~\ref{sec2}. Since the latter turn out to be the extreme singular values of \scalebox{0.88}{$\mathfrak{A}:=\mathfrak{H}^{\top}\in\rNn$} (assuming \scalebox{0.88}{$N\geq n$}), the remainder of the analysis focuses on
{\footnotesize{
\begin{equation}\label{AMatrixEnsemble}
\mathfrak{A}=\frac{1}{\sqrt{n}}(\Xij)_{\underset{1\leq j\leq n}{1\leq i\leq N}}=\frac{\hat{q}}{\sqrt{n}}(\Yij)_{\underset{1\leq j\leq n}{1\leq i\leq N}},\ \mbox{\normalsize with}\ \Xij\sim\bbmu_X(q),\ \mbox{\normalsize and}\ \Yij\sim\bbmu_Y(q)
\end{equation}
}}so that the particular case where $\frac{1}{q}=\frac{s}{2n}$ leads to
{\footnotesize{
\begin{equation}\label{AhashingLike}
\mathfrak{A}=\frac{1}{\sqrt{s}}(\Yij)_{\underset{1\leq j\leq n}{1\leq i\leq N}},\ \mbox{\normalsize where}\ \Yij\ \mbox{\normalsize is distributed as}\ Y\ \mbox{\normalsize in Proposition~\ref{sProposition}}.
\end{equation}
}}
Recall the following useful definition (see, e.g.,~\cite[Definition~2.2]{Rudelson2014LectNotes}), which is considered throughout the manuscript.
\begin{definition}\label{singValDef}
The singular values of \scalebox{0.88}{$\Amatrix\in\rNn,\ N\geq n$} are the eigenvalues of \scalebox{0.88}{$\abs{\Amatrix}:=\sqrt{\Amatrix^{\top}\Amatrix}$}, arranged in the decreasing  order \scalebox{0.88}{$s_1(\Amatrix)\geq s_2(\Amatrix)\geq \dots \geq s_n(\Amatrix)$}.
\end{definition}
The first singular value is related to the operator norm of the linear operator $\Amatrix$ acting between Euclidean spaces~\cite{Rudelson2014LectNotes,RudVersh2010ExtrSingVal} as follows: \scalebox{0.88}{$s_1(\Amatrix) = \underset{\x\in \snOne}{\max} \normiiN{\Amatrix\x}=\norme{\Amatrix:\rn\to\R^N}$}. On the other hand, \scalebox{0.88}{$s_n(\Amatrix) = \underset{\x\in \snOne}{\min} \normiiN{\Amatrix\x}$}.

\subsection{Asymptotic behavior of extreme singular values}
Let \scalebox{0.8}{$\mathfrak{X}\in \rNn,\ N\geq n$}, be a matrix whose i.i.d. entries \scalebox{0.88}{$\Xij$} are mean zero and unit variance random variables. It is well known from the Marchenko--Pastur law~\cite{MarchenKoPastur1967} that most singular values of \scalebox{0.88}{$\mathfrak{X}$} belong to \scalebox{0.88}{$[\sqrt{N}-\sqrt{n}, \sqrt{N}+\sqrt{n}]$}; and as specified in~\cite{RudVersh2010ExtrSingVal}, it is  true that under mild assumptions, all of them lie in the above interval so that asymptotically
{\footnotesize{
\begin{equation}\label{asymptotics21}
s_n(\mathfrak{X})\sim \sqrt{N}-\sqrt{n} \quad\mbox{\normalsize and}\quad s_1(\mathfrak{X})\sim \sqrt{N}+\sqrt{n},
\end{equation}
}}
a universal fact, holding for general distributions. The result was proved for \scalebox{0.88}{$s_1(\mathfrak{X})$} by Geman~\cite{Geman1980limit} and Yin, Bai, and Krishnaiah~\cite{YinBaiKrish1988limit}. It was proved for \scalebox{0.88}{$s_n(\mathfrak{X})$} by Silverstein~\cite{SylversteinSnval1985} when \scalebox{0.88}{$\mathfrak{X}$} is a Gaussian random matrix. A unified treatment of both \scalebox{0.88}{$s_n(\mathfrak{X})$} and \scalebox{0.88}{$s_1(\mathfrak{X})$} was given by Bai and Yin~\cite[Theorems~1 \&~2]{BaiYinSmallCovEig1993} for general distributions through the following theorem.

\begin{theorem}\label{baiyinTheorems}
Let \scalebox{0.88}{$\mathfrak{X}=\mathfrak{X}_{N,n}\in \rNn$} with \scalebox{0.88}{$N\geq n$} be a matrix whose entries \scalebox{0.88}{$\Xij$} are i.i.d. random variables with mean zero, unit variance, and fourth moment finite; and let \scalebox{0.88}{$\mathfrak{S}:=\frac{1}{N}\mathfrak{X}^{\top}\mathfrak{X}$}. Then as \scalebox{0.88}{$N\to \infty$, $n\to \infty$} and the ``aspect ratio'' \scalebox{0.88}{$n/N\to \mathfrak{y}\in (0, 1)$}, the extreme eigenvalues of \scalebox{0.88}{$\mathfrak{S}$} satisfy
{\footnotesize{
\[\lim  \uplambda_{\min}(\mathfrak{S}) = (1-\sqrt{\mathfrak{y}})^2\ \mbox{\normalsize almost surely}\quad\mbox{\normalsize and}\quad \lim  \uplambda_{\max}(\mathfrak{S}) = (1+\sqrt{\mathfrak{y}})^2\ \mbox{\normalsize almost surely}.\]
}}
\end{theorem}

It easily follows from Theorem~\ref{baiyinTheorems} that (see also~\cite[Theorem~2.1]{RudVersh2010ExtrSingVal}) when \scalebox{0.88}{$n, N$} grow to infinity while the aspect ratio \scalebox{0.88}{$n/N\to \mathfrak{y}\in (0, 1)$}, then \scalebox{0.88}{$\frac{1}{\sqrt{N}}s_n(\mathfrak{X})\to 1-\sqrt{\mathfrak{y}}$} and \scalebox{0.88}{$\frac{1}{\sqrt{N}}s_1(\mathfrak{X})\to 1+\sqrt{\mathfrak{y}}$}. An analogous result for hashing-like matrices is stated next {\bl and illustrated by Figure~\ref{lab01}}.

%\twocolumn 

% \begin{figure}
% \centering
% \includegraphics[scale=0.5]{figs/sn_hashing_100}
% \includegraphics[scale=0.5]{figs/s1_hashing_100}
% \caption{\small{Graphs  from the simulation of the asymptotic results of Corollary~\ref{CorBaiYin} (applied to $n\times N$ $s$-hashing matrices as detailed in Section~\ref{sec44}). Each shaded region contains the convergence graphs of $100$ sequences of singular values (smallest at top, largest at bottom) of the above matrices. In each region, the convergence of the sequence of averages of the $N$th terms of the above sequences is illustrated by a red curve. For each sequence of matrices considered, the {\it aspect ratio} $n/N$ satisfies $n/N\to \mathfrak{y}:=1/100$ as $n$ and $N$ get larger (while $s/n$ remains constant, approximately). As predicted by theory, the (almost sure) limit of the sequence of smallest singular values is $\frac{1}{\sqrt{\mathfrak{y}}}-1=9$, illustrated by the dotted line. This limit is $\frac{1}{\sqrt{\mathfrak{y}}}+1=11$ for the sequence of largest singular values.}}
% \label{lab0}
% \end{figure}

\begin{figure}
\centering
\includegraphics[scale=0.163]{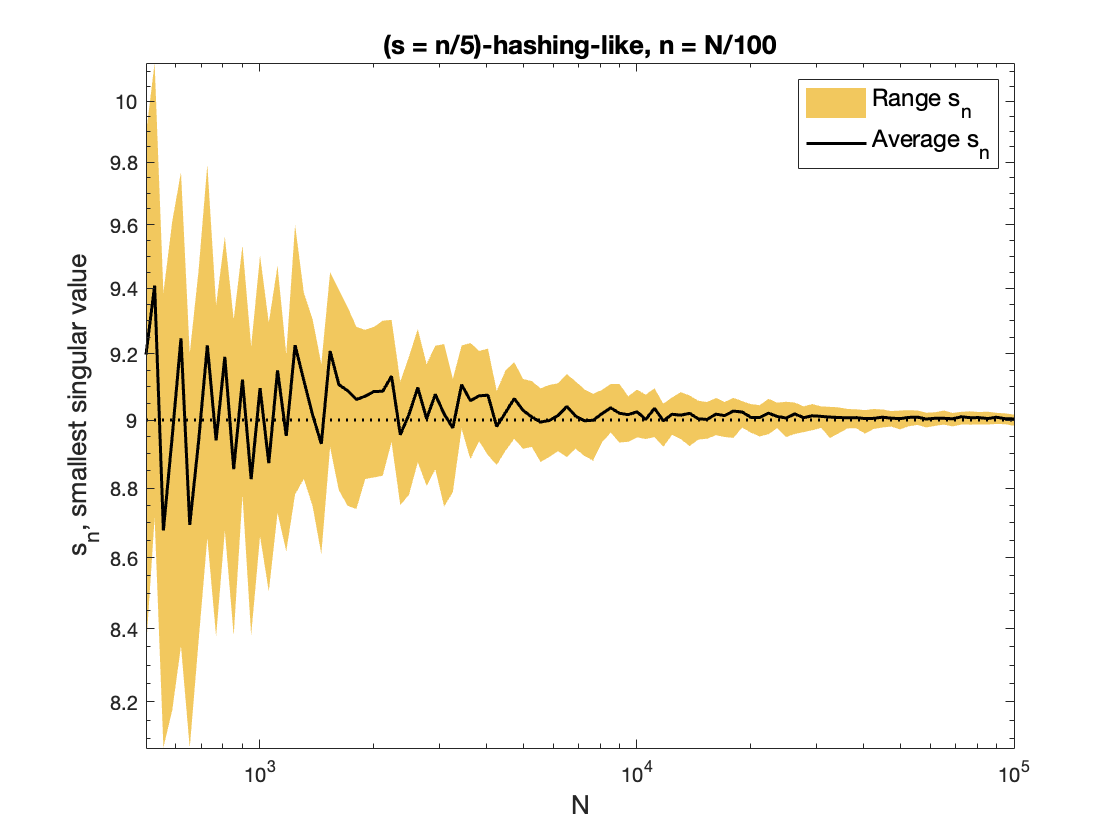}
\includegraphics[scale=0.163]{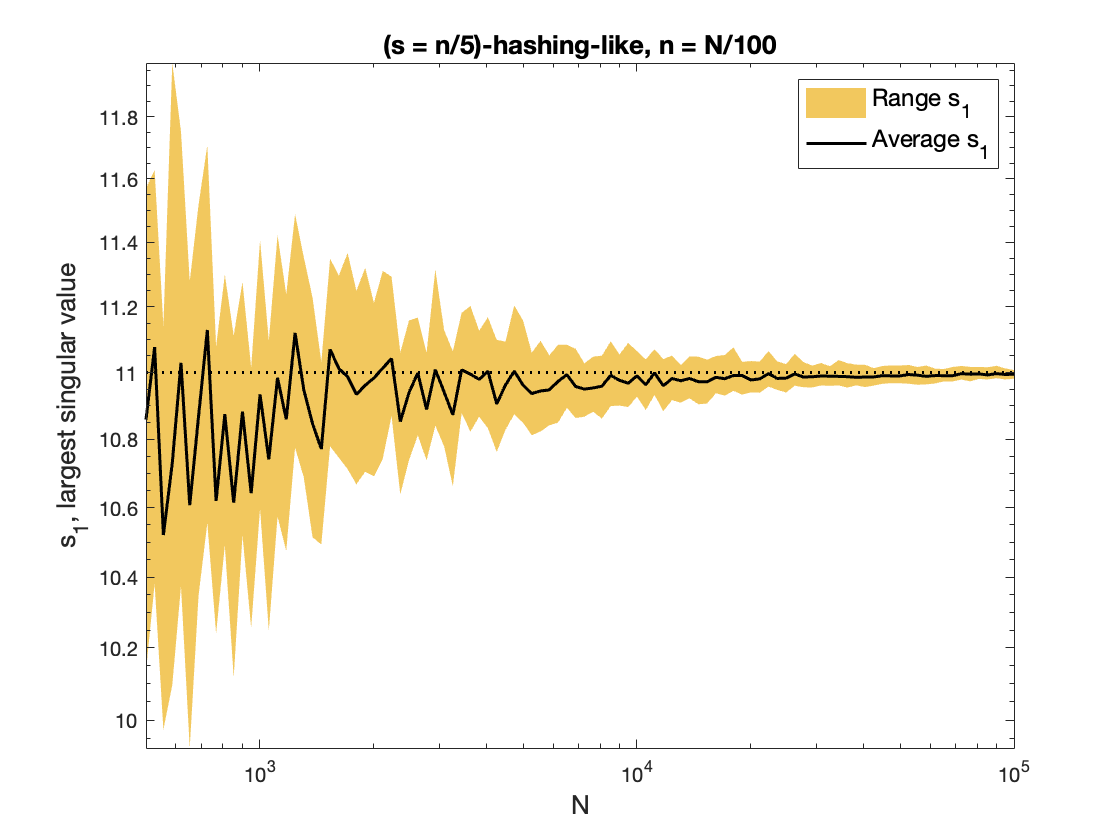}
\caption{\small{Graphs from the simulation of the asymptotic results of Corollary~\ref{CorBaiYin} (detailed in Appendix). Each shaded region contains the convergence graphs of $100$ sequences of singular values (smallest on the left, largest on the right) of $n\times N$ $s$-hashing-like matrices. In each region, the convergence of the sequence of averages of the $N$th terms of the above sequences is illustrated by a black curve. For each sequence of matrices considered, the {\it aspect ratio} $n/N$ satisfies $n/N\to \mathfrak{y}:=1/100$ as $n$ and $N$ get larger (while $s/n$ remains constant, approximately). As predicted by theory, the (almost sure) limit of the sequence of smallest singular values is $\frac{1}{\sqrt{\mathfrak{y}}}-1=9$, illustrated by the dotted line. This limit is $\frac{1}{\sqrt{\mathfrak{y}}}+1=11$ for the sequence of largest singular values.\\ $ $}}
\label{lab01}
\end{figure}
%\onecolumn
\begin{corollary}\label{CorBaiYin}
Consider the random matrix \scalebox{0.85}{$\mathfrak{A}={\bbl \mathfrak{A}_{N,n}}=\frac{1}{\sqrt{s}}(\Yij)_{\underset{1\leq j\leq n}{1\leq i\leq N}}\in \rNn$}\\ given by~\eqref{AhashingLike}.
Then as \scalebox{0.88}{$N\to \infty$}, \scalebox{0.88}{$n\to \infty$}, and \scalebox{0.88}{$n/N\to \mathfrak{y}\in (0, 1)$},
{\footnotesize{
\begin{equation}\label{singLimits}
\lim  s_n(\Amatrix)=\frac{1}{\sqrt{\mathfrak{y}}}-1\quad\mbox{\normalsize almost surely}\qquad\mbox{\normalsize and}\qquad \lim  s_1(\Amatrix)=\frac{1}{\sqrt{\mathfrak{y}}}+1\quad\mbox{\normalsize almost surely},
\end{equation}
}}
{\bbl if} \scalebox{0.88}{$s/n$} remains constant.
\end{corollary}

\begin{proof}
We observe that  \scalebox{0.88}{$\Amatrix=\frac{1}{\sqrt{n}}\left(\sqrt{\frac{n}{s}}\Yij\right)_{\underset{1\leq j\leq n}{1\leq i\leq N}}=\frac{1}{\sqrt{n}}(\Xij)_{\underset{1\leq j\leq n}{1\leq i\leq N}}$}, with \scalebox{0.86}{$\Xij\!~\sim \bbmu_X(2n/s)$}. Let \scalebox{0.88}{$\mathfrak{X}:=(\Xij)_{\underset{1\leq j\leq n}{1\leq i\leq N}}=\sqrt{n}\Amatrix$}, and assume that $s/n$ is fixed, say, \scalebox{0.88}{$s/n=2/q$} for some fixed \scalebox{0.88}{$q\geq 2$}, thus preventing the distribution \scalebox{0.88}{$\bbmu_X(q)$} of the entries \scalebox{0.88}{$\Xij$} from being dependent of \scalebox{0.88}{$n$} and \scalebox{0.88}{$s$}. Since the \scalebox{0.88}{$\Xij$} are mean zero, unit variance, with fourth moment finite,  it follows from Theorem~\ref{baiyinTheorems} with \scalebox{0.8}{$\mathfrak{S}:=\left(\frac{\ds{\mathfrak{X}}}{\sqrt{N}}\right)^{\top}\left(\frac{\ds{\mathfrak{X}}}{\sqrt{N}}\right)$} that \scalebox{0.88}{$\lim \uplambda_{\min}(\mathfrak{S}) = \lim \left(s_n(\mathfrak{X}/\sqrt{N})\right)^2= \lim \left(\sqrt{\frac{n}{N}}s_n(\Amatrix)\right)^2=(1-\sqrt{\mathfrak{y}})^2$} almost surely, whence  almost surely, \scalebox{0.88}{$\lim  s_n(\Amatrix)=\frac{1}{\sqrt{\mathfrak{y}}}(1-\sqrt{\mathfrak{y}})$}. The result for \scalebox{0.88}{$s_1(\Amatrix)$} follows by similar arguments.
\end{proof}

{\bl
\begin{corollary}\label{CorBaiYin2}
Let  \scalebox{0.88}{$\mathfrak{A}=\mathfrak{A}_{N,n}=\frac{1}{\sqrt{n}}(\Xij)_{\underset{1\leq j\leq n}{1\leq i\leq N}}\in \rNn$} be given by~\eqref{AMatrixEnsemble}. Then as \scalebox{0.87}{$N\to \infty$}, \scalebox{0.88}{$n\to \infty$}, and \scalebox{0.88}{$n/N\to \mathfrak{y}\in (0, 1)$}, \scalebox{0.88}{$s_n(\Amatrix)$} and \scalebox{0.88}{$s_1(\Amatrix)$} satisfy~\eqref{singLimits}.
\end{corollary}
\begin{proof}
The proof trivially follows from that of Corollary~\ref{CorBaiYin}.
\end{proof}

\begin{figure}[ht!]
\centering
\includegraphics[scale=0.43]{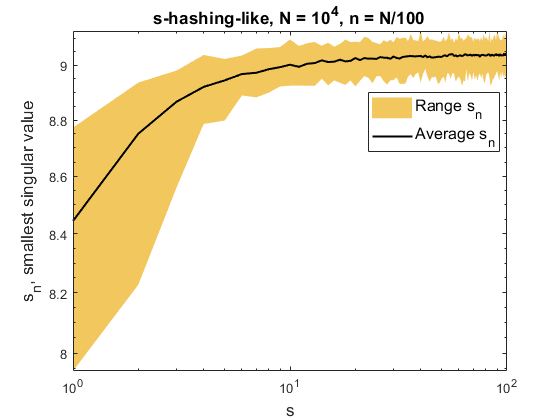}
\includegraphics[scale=0.43]{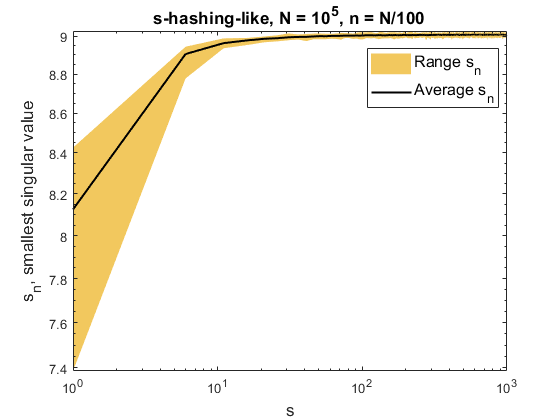}
\includegraphics[scale=0.43]{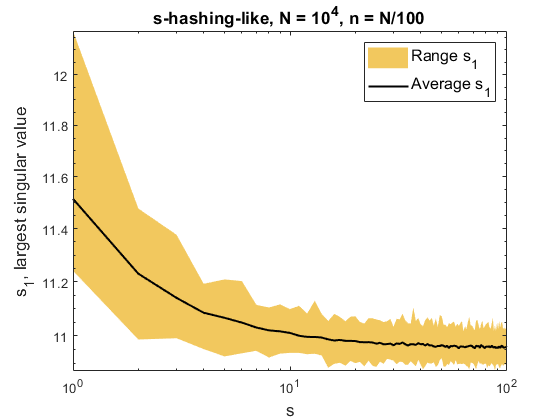}
\includegraphics[scale=0.43]{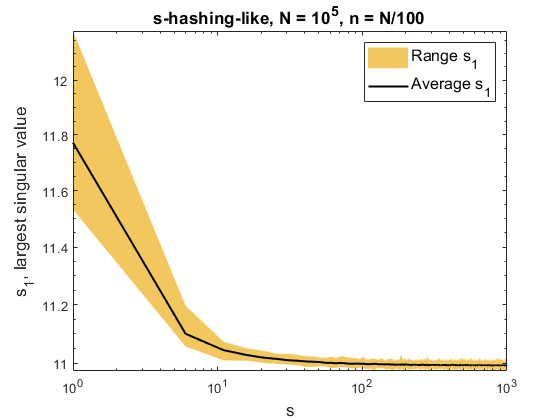}
\caption{\small{\bbl Graphs illustrating the behavior of the extreme singular values of $s$-hashing-like matrices with fixed dimensions, with respect to~$s$. Each shaded region contains the convergence graphs of $100$ sequences of singular values (smallest at the top, largest at the bottom) of $n\times N$ matrices, for various dimensions ($N=10^4$ on the left and $N=10^5$ on the right, with $n/N=1/100:=\mathfrak{y}$). In each region, the behavior of the sequence of averages of the $s$th terms of the above sequences is illustrated by a black curve. The increase of~$s$ reduces the sparsity level so that at the limit $s=n$, the matrix entries $X_{ij}/\sqrt{n}$ are given by~\eqref{AMatrixEnsemble} with $X_{ij}\sim \bbmu_X(2)$. Thus, Corollary~\ref{CorBaiYin2} predicts that the limiting extreme singular values approach $\frac{1}{\sqrt{\mathfrak{y}}}-1=9$ (smallest) and $\frac{1}{\sqrt{\mathfrak{y}}}+1=11$ (largest). On the other hand as expected, the smaller~$s$ is, the more singular the matrices, the smaller the smallest singular value, and the larger the largest singular value.}}
\label{lab10}
\end{figure}

{\bbl The behavior of the extreme singular values of $s$-hashing-like matrices with fixed dimensions $n$ and $N$, with respect to~$s$ is depicted in Figure~\ref{lab10}, illustrating in particular the singularity of the matrices with respect to the sparsity parameter~$s$ and the connection of the limiting extreme singular values to the results of Corollary~\ref{CorBaiYin2}.
}

%For each sequence of matrices considered, the {\it aspect ratio} $n/N$ satisfies $n/N\to \mathfrak{y}:=1/100$ as $n$ and $N$ get larger (while $s/n$ remains constant, approximately). As predicted by theory, the (almost sure) limit of the sequence of smallest singular values is $\frac{1}{\sqrt{\mathfrak{y}}}-1=9$, illustrated by the dotted line. This limit is $\frac{1}{\sqrt{\mathfrak{y}}}+1=11$ for the sequence of largest singular values.\\ $ $
The remainder of this section is devoted to showing that under the conditions of Corollaries~\ref{CorBaiYin} and~\ref{CorBaiYin2}, \scalebox{0.88}{$s_n(\Amatrix)$} and \scalebox{0.88}{$s_1(\Amatrix)$} converge respectively, in distribution to the Tracy--Widom \scalebox{0.88}{$TW_1$} law (recalled below), after proper rescalings, making use of existing results due to P\'ech\'e~\cite{Peche2009Universality} and Feldheim and Sodin~\cite{felSodTracyWid2010}. As recalled in~\cite{felSodTracyWid2010}, it has been long conjectured that some of the known asymptotic statistical properties of large matrices with Gaussian entries such as those related to the local statistics of the eigenvalues at the edge of the spectrum, should be valid, in particular, for more general random matrices with independent entries. One of the first rigorous results of this kind (which is part of a phenomenon referred to as {\it universality} in the physics literature) due to Soshnikov~\cite{Soshnikov1999UnivSpectr} (see also~\cite[Theorem I.1.3]{felSodTracyWid2010}), was later extended in~\cite{SoshnikovTracyWid2002} to the largest eigenvalues of sample covariance matrices \scalebox{0.88}{$\mathfrak{X}^{\top}\mathfrak{X}$} under some restrictions on the
dimensions of the matrices that were then disposed of by P\'ech\'e~\cite{Peche2009Universality} in Theorem~\ref{soshpecheLarge} below. Let Ai denote the standard Airy function (see, e.g.,~\cite[Definition~I.4.5]{felSodTracyWid2010}) uniquely defined by \scalebox{0.88}{$\text{Ai}''(x)=x\text{Ai}(x)$} with \scalebox{0.88}{$\text{Ai}(x)\sim\frac{1}{2\sqrt{\pi}x^{1/4}}\exp\left(-\frac{2}{3}x^{3/2}\right)$} as \scalebox{0.88}{$x\to \infty$}. We first recall the {\it Gaussian Orthogonal Ensemble} (GOE, see, e.g.,~\cite[Example~I.3.2]{felSodTracyWid2010}) Tracy--Widom distribution \scalebox{0.88}{$TW_1$} (see, e.g.,~\cite[Definition~1.1]{Peche2009Universality}), defined by the cumulative distribution function \scalebox{0.88}{$F_1(x):= \exp\accolade{\int_x^{\infty} -\frac{q(t)}{2} + \frac{(x-t)}{2}q^2(t) dt}$}, where $q(\cdot)$ is the so-called Hastings--McLeod solution of the Painlev\'e~II differential equation \scalebox{0.88}{$q''(x) = xq(x)+2q^3(x)$} with boundary condition \scalebox{0.88}{$q(x)\sim\text{Ai}(x)$} as \scalebox{0.88}{$x\to\infty$}.

\begin{theorem}\label{soshpecheLarge}(Soshnikov~\cite{SoshnikovTracyWid2002}, P\'ech\'e~\cite{Peche2009Universality}, \cite[Theorem I.1.2]{felSodTracyWid2010})
Let \scalebox{0.88}{$\accolade{\mathfrak{X}_{N,n}}$}, with \scalebox{0.88}{$\mathfrak{X}_{N,n}\in \rNn$}, \scalebox{0.88}{$N\geq n$}, be a sequence of random matrices whose independent entries have a symmetric distribution, sub-Gaussian tails and unit variance. Denote by \scalebox{0.88}{$\lambda_1:=\lambda_{\max}\left({\mathfrak{X}_{N,n}}^{\top}\mathfrak{X}_{N,n}\right)$} the largest eigenvalue of \scalebox{0.88}{${\mathfrak{X}_{N,n}}^{\top}\mathfrak{X}_{N,n}$}. Then as \scalebox{0.88}{$N\to \infty$}, \scalebox{0.88}{$n\to \infty$} and \scalebox{0.88}{$n/N\to \mathfrak{y}\in (0, 1)$}, \scalebox{0.88}{$\xi_{N,n}:=\frac{\lambda_1-\left(N^{1/2}+n^{1/2}\right)^2}{\left(N^{1/2}+n^{1/2}\right)\left(N^{-1/2}+n^{-1/2}\right)^{1/3}}$} converges in distribution to the Tracy--Widom \scalebox{0.88}{$TW_1$} law.
\end{theorem}

The following result {\bl illustrated by Figure~\ref{lab02}} provides the exact asymptotic distribution of $s_1(\Amatrix)$ after a proper rescaling. 
\begin{figure}
\centering
\includegraphics[scale=0.43]{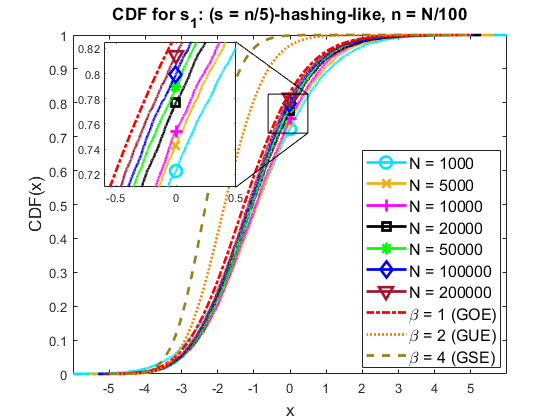}
\includegraphics[scale=0.43]{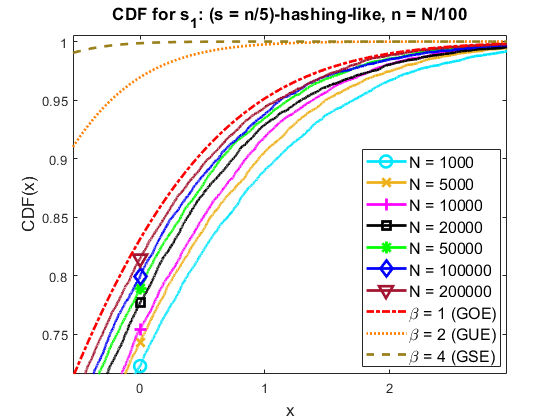}
\caption{\small{Cumulative distribution functions (CDFs) illustrating the result of Corollary~\ref{s1Distribution}, that is, the convergence in distribution to the GOE Tracy--Widom $TW_1$ law $(\beta=1)$, of the rescaled largest singular values of $n\times N$ $s$-hashing-like matrices. For each $N$, the curve of the empirical CDF is generated using $10^4$ matrix samples. The larger $N, n$, the closer the corresponding curve is to that of $TW_1$ and not those of the Gaussian unitary and the Gaussian symplectic ensembles, denoted by GUE $(\beta=2)$ and GSE $(\beta=4)$, respectively.}}
\label{lab02}
\end{figure}
\begin{corollary}\label{s1Distribution}
Under the conditions of any of Corollaries~\ref{CorBaiYin} and~\ref{CorBaiYin2}, the following holds in distribution: \\ \scalebox{0.86}{$2\sqrt{\mathfrak{y}}\left(1+1/\sqrt{\mathfrak{y}}\right)^{-1/3}N^{2/3}\left(s_1(\Amatrix)-\left(1+\sqrt{N/n}\right)\right)\to\  TW_1$} as \scalebox{0.88}{$N\to \infty$}, \scalebox{0.88}{$n\to \infty$} and \scalebox{0.88}{$n/N\to \mathfrak{y}\in (0, 1)$}.

\end{corollary}

\begin{proof}
It is easy to observe that the entries of \scalebox{0.88}{$\mathfrak{X}_{N,n}:=\sqrt{n}\Amatrix$} trivially satisfy the conditions of Theorem~\ref{soshpecheLarge}, and that \scalebox{0.88}{$\lambda_1:=\lambda_{\max}\left({\mathfrak{X}_{N,n}}^{\top}\mathfrak{X}_{N,n}\right)=n(s_1(\Amatrix))^2$}. Moreover,
{\footnotesize{
\begin{equation*}
\xi_{N,n} = \frac{n(s_1(\Amatrix))^2-\left(N^{1/2}+n^{1/2}\right)^2}{\left(N^{1/2}+n^{1/2}\right)\left(N^{-1/2}+n^{-1/2}\right)^{1/3}}=\frac{(s_1(\Amatrix))^2-\left(1+\sqrt{N/n}\right)^2}{\sqrt{N/n}\left(1+\sqrt{N/n}\right)^{4/3}}N^{2/3}=: \xi_{N,n}^{(1)}\xi_{N,n}^{(2)},
\end{equation*}
}}
where \scalebox{0.88}{$\xi_{N,n}^{(1)}:=\frac{s_1(\Amatrix)+\left(1+\sqrt{N/n}\right)}{\sqrt{N/n}\left(1+\sqrt{N/n}\right)^{4/3}}$} and \scalebox{0.88}{$\xi_{N,n}^{(2)}:=N^{2/3}\left(s_1(\Amatrix)-\left(1+\sqrt{N/n}\right)\right)$}. It follows from any of Corollaries~\ref{CorBaiYin} and~\ref{CorBaiYin2} that almost surely, \scalebox{0.88}{$\xi_{N,n}^{(1)}\to 2\sqrt{\mathfrak{y}}\left(1+1/\sqrt{\mathfrak{y}}\right)^{-1/3}=:\tilde{\mathfrak{y}}$}, that is, to a constant, as \scalebox{0.88}{$N\to \infty$, $n\to \infty$} and \scalebox{0.88}{$n/N\to \mathfrak{y}$}. On the other hand, it follows from Theorem~\ref{soshpecheLarge} that in distribution, \scalebox{0.88}{$\xi_{N,n}\to \xi\sim TW_1$}, whence in distribution (see, e.g.,~\cite[Exercise~3.2.14]{durrett2019probability}) \scalebox{0.88}{$\ \xi_{N,n}^{(2)} = \xi_{N,n}\left(\xi_{N,n}^{(1)}\right)^{-1}\to {\tilde{\mathfrak{y}}}^{-1}\xi$}, which completes the proof.
\end{proof}

To derive the asymptotic distribution of the rescaled smallest singular value (illustrated by Figure~\ref{lab03}), we need the following result from~\cite{felSodTracyWid2010}.

\begin{theorem}\label{FeldSod2010Theor}(\cite[Theorem~I.1.1]{felSodTracyWid2010})
Let the same assumptions in Theorem~\ref{soshpecheLarge} hold, and let \scalebox{0.88}{$\lambda_n:=\lambda_{\min}\left({\mathfrak{X}_{N,n}}^{\top}\mathfrak{X}_{N,n}\right)$} be the smallest eigenvalue of \scalebox{0.87}{${\mathfrak{X}_{N,n}}^{\top}\mathfrak{X}_{N,n}$}. Then as \scalebox{0.87}{$N\to \infty$}, \scalebox{0.87}{$n\to \infty$} and \scalebox{0.87}{$n/N\to \mathfrak{y}\in (0, 1)$}, \scalebox{0.87}{$\zeta_{N,n}:=\frac{\lambda_n-\left(n^{1/2}-N^{1/2}\right)^2}{\left(n^{1/2}-N^{1/2}\right)\left(n^{-1/2}-N^{-1/2}\right)^{1/3}}$} converges in distribution to the Tracy--Widom \scalebox{0.88}{$TW_1$} law.
\end{theorem}

\begin{corollary}\label{snDistribution}
Under the conditions of any of Corollaries~\ref{CorBaiYin} and~\ref{CorBaiYin2}, the following holds in distribution:\\ \scalebox{0.88}{$-2\sqrt{\mathfrak{y}}\left(1/\sqrt{\mathfrak{y}}-1\right)^{-1/3}N^{2/3}\left(s_n(\Amatrix)+\left(1-\sqrt{N/n}\right)\right)\to\  TW_1$} as \scalebox{0.85}{$N\to \infty$}, \scalebox{0.88}{$n\to \infty$} and \scalebox{0.88}{$n/N\to \mathfrak{y}\in (0, 1)$}.
\end{corollary}

\begin{proof}
The proof follows that of Corollary~\ref{s1Distribution} by observing that in distribution, \scalebox{0.88}{$\zeta_{N,n}=\zeta_{N,n}^{(1)}\zeta_{N,n}^{(2)}\to\zeta\sim TW_1$} thanks to Theorem~\ref{FeldSod2010Theor}, where \scalebox{0.88}{$\zeta_{N,n}^{(1)}:=\frac{\left(1-\sqrt{N/n}\right)-s_n(\Amatrix)}{\sqrt{N/n}\left(\sqrt{N/n}-1\right)^{4/3}}\to$} \scalebox{0.88}{$-2\sqrt{\mathfrak{y}}\left(1/\sqrt{\mathfrak{y}}-1\right)^{-1/3}=:\hat{\mathfrak{y}}$} almost surely thanks to any of Corollaries~\ref{CorBaiYin} and~\ref{CorBaiYin2}, and  \scalebox{0.88}{$\zeta_{N,n}^{(2)}:=N^{2/3}\left(s_n(\Amatrix)+\left(1-\sqrt{N/n}\right)\right)= \zeta_{N,n}\left(\zeta_{N,n}^{(1)}\right)^{-1}\to {\hat{\mathfrak{y}}}^{-1}\zeta$} in distribution.
\end{proof}

}

\begin{figure}
\centering
\includegraphics[scale=0.43]{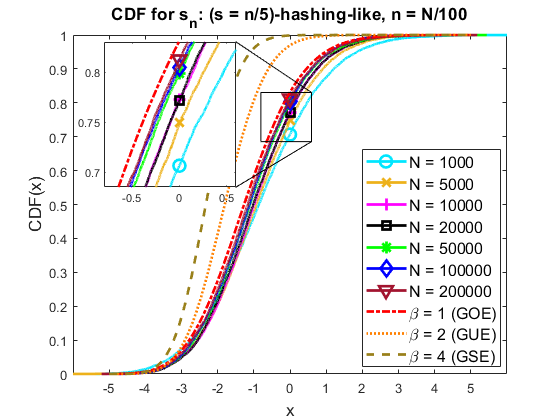}
\includegraphics[scale=0.43]{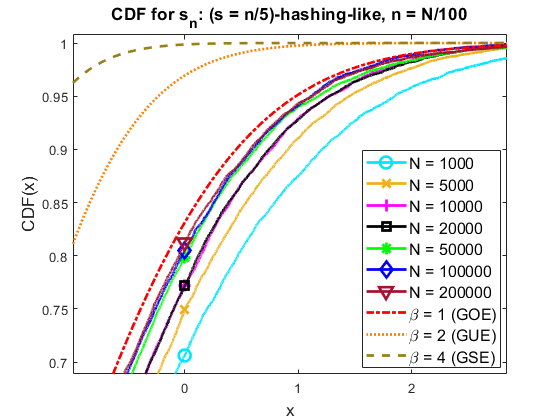}
\caption{\small{Cumulative distribution functions (CDFs) illustrating the result of Corollary~\ref{snDistribution}, that is, the convergence in distribution to the GOE Tracy--Widom $TW_1$ law $(\beta=1)$, of the rescaled smallest singular values of $n\times N$ $s$-hashing-like matrices. For each $N$, the curve of the empirical CDF is generated using $10^4$ matrix samples. The larger $N, n$, the closer the corresponding curve is to that of $TW_1$ and not those of the  Gaussian unitary and the Gaussian symplectic ensembles, denoted by GUE $(\beta=2)$ and GSE $(\beta=4)$, respectively.}}
\label{lab03}
\end{figure}

\subsection{Non-asymptotic behavior of extreme singular values}

As mentioned in~\cite{RudVersh2010ExtrSingVal}, the extent to which the limiting behavior of the smallest and largest singular values  such as asymptotics~\eqref{asymptotics21} manifests itself in fixed dimensions is not entirely clear. 
%Given a random matrix $\Amatrix$, one may want to place useful lower an upper bounds on it in order to control its magnitude
Moreover, understanding the behavior of extreme singular values of random matrices is crucial in myriad fields{\bl .} 
%including numerical linear algebra where the  condition number \scalebox{0.88}{$\kappa(\Amatrix):=s_1(\Amatrix)/s_n(\Amatrix)$} is often used as a measure of stability of matrix algorithms. On the other hand, 
{\bl For example,} in the randomized direct-search algorithmic framework recently introduced in~\cite{RobRoy2022RedSpace}, the smallest singular value of JLTs used to select random subspaces in which are defined  polling directions were required to be lower bounded by a fixed quantity with high probability. We also note that in all the algorithmic frameworks of~\cite{CRsubspace2021,DzaWildSub2022,RobRoy2022RedSpace}, random subspaces are defined as the range of \scalebox{0.88}{$\Amatrix\in\rNn$} so that determining their dimension requires estimating \scalebox{0.88}{$s_n(\Amatrix)$} non-asymptotically. The remainder of this section therefore focuses on establishing upper and lower bounds on \scalebox{0.88}{$s_1(\Amatrix)$ and $s_n(\Amatrix)$}, respectively, where \scalebox{0.88}{$\Amatrix=\Hfrac^{\top}\in \rNn, N\geq n$}, that is, the transpose of the matrices analyzed in Section~\ref{sec2}, since they can define JLTs. 

The proofs of the main results use a useful and important technique in geometric functional analysis: the $\varepsilon$-{\it net argument}, consisting of discretizations of unit spheres through so-called $\varepsilon$-nets in the Euclidean norm, which are subsets of these spheres that approximate each of their points up to error $\varepsilon$. 
%an $\varepsilon$-{\it net argument}, a useful and very important technique in geometric functional analysis~\cite{RudVersh2010ExtrSingVal}.
%The proof of~\cite[Proposition~2.4]{RudVersh2010ExtrSingVal} uses discretizations of unit spheres through so-called $\varepsilon$-nets in the Euclidean norm, which are subsets of these spheres that approximate every point of the spheres up to error $\varepsilon$. 
A more precise and general definition of $\varepsilon$-nets from~\cite[Definition~4.1]{Rudelson2014LectNotes} is provided next.
\begin{definition}\label{netDefinition}
Let \scalebox{0.88}{$(\mathscr{M},d)$} be a metric space. A set \scalebox{0.88}{$\n\subset \mathscr{M}$} is called an \scalebox{0.88}{$\varepsilon$}-net for \scalebox{0.88}{$\mathscr{K}\subset \mathscr{M}$}, for some \scalebox{0.88}{$\varepsilon>0$}, if \scalebox{0.88}{$\forall x\in \mathscr{K}\ \ \exists y\in\n\, \mbox{such that} \ \ d(x, y)<\varepsilon.$}
\end{definition}
The next result was derived in the proof of~\cite[Lemma~4.3]{Rudelson2014LectNotes}, which was used in~\cite[Proposition~4.4]{Rudelson2014LectNotes} and~\cite[Proposition~2.4]{RudVersh2010ExtrSingVal} to estimate the largest singular value of sub-Gaussian random matrices. It will be crucial for estimating the largest and smallest singular values of {\bl \scalebox{0.88}{$\Amatrix$}} in Theorems~\ref{LargestSingResult} and~\ref{SmallestSingResult}, respectively.
\begin{lemma}\label{volumLemma}(Volumetric estimate). 
For any \scalebox{0.88}{$\varepsilon\in (0,1)$}, there exists an \scalebox{0.88}{$\varepsilon$}-net \scalebox{0.88}{$\n_\varepsilon\in\mathbb{S}^{m-1}$} such that 
\scalebox{0.88}{$\abs{\n_\varepsilon}\leq \left(1+\frac{2}{\varepsilon}\right)^m.$}
\end{lemma}
Using $\varepsilon$-nets, we prove a basic bound on the first singular value of {\bl \scalebox{0.88}{$\Amatrix$}} in {\bl Theorem~\ref{LargestSingResult} and} Corollary~\ref{LargestSingCor} {\bl below}.
\begin{theorem}\label{LargestSingResult}(Largest singular value).
Consider the matrix \scalebox{0.88}{$\Amatrix\in \rNn$} of~\eqref{AMatrixEnsemble}.  
{\bl There exists an absolute constant \scalebox{0.88}{$\tilde{c}_{2,3}\in (0,1)$} such that} for \scalebox{0.88}{$N$} sufficiently large,
{\small{
\begin{equation}\label{largestSingValBound}
\prob{s_1(\Amatrix)>\hat{q}t\sqrt{\frac{N}{n}}}\leq e^{-c_0Nt^2} \,  \forall\ t\geq C_0{\bl ,}
%SW: I was trying to inline but it did not fit
%\quad \mbox{where } c_0=\tilde{c}_{2,3}^2\ln{\left(\frac{q}{2}+1\right)} 
%\mbox{ and } C_0=\left(\frac{\ln 5}{c_0}\right)^{1/2}.
\end{equation}
}}
{\bl where \scalebox{0.88}{$c_0=\tilde{c}_{2,3}^2\ln{\left(\frac{q}{2}+1\right)}$} and \scalebox{0.88}{$C_0=\left(\frac{\ln 5}{c_0}\right)^{1/2}$}.}
\end{theorem}

\begin{proof}
The proof is inspired by that of~\cite[Proposition~4.4]{Rudelson2014LectNotes}. 
Recall that the independent entries \scalebox{0.88}{$\accolade{\Amatrix_{ij}}_{\underset{1\leq j\leq n}{1\leq i\leq N}}$} of \scalebox{0.88}{$\Amatrix$} are given by \scalebox{0.88}{$\Amatrix_{ij} = \frac{1}{\sqrt{n}} \Xij=\frac{1}{\sqrt{n}}\hat{q}\Yij$}, where the \scalebox{0.88}{$\Yij$} are independent copies of the variable \scalebox{0.88}{$Y$} in Proposition~\ref{HoeffdingTypeIneq}.
Let $\n$ and $\M$  be \scalebox{0.88}{$\frac{1}{2}$}-nets of \scalebox{0.88}{$\sNOne$} and \scalebox{0.88}{$\snOne$}, respectively. For all \scalebox{0.88}{$\ub\in\snOne$}, there exists \scalebox{0.88}{$\x\in\M$} such that \scalebox{0.88}{$\normiin{\x-\ub}<\frac{1}{2}$} by Definition~\ref{netDefinition}, which implies  %\scalebox{0.88}{
{\bbl 
\[\normiiN{\Amatrix \ub}=\normiiN{\Amatrix(\x-(\x-\ub))}\leq\normiiN{\Amatrix \x} +\norme{\Amatrix}\normiin{\x-\ub} \leq\normiiN{\Amatrix \x}+\frac{1}{2}\norme{\Amatrix}.\]
}
%}. 
Hence, 
%{\footnotesize{
\begin{equation}\label{supsup1}
\norme{\Amatrix}\leq 2\underset{\x\in\M}{\sup}\normiiN{\Amatrix \x}=2\underset{\x\in\M}{\sup}\ \underset{\vi\in\sNOne}{\sup}\abs{\scal{\Amatrix \x}{\vi}}.
\end{equation}
%}}
On the other hand, let \scalebox{0.88}{$\y\in\n$} be such that \scalebox{0.88}{$\normiiN{\y-\vi}<\frac{1}{2}$} thanks to Definition~\ref{netDefinition}. From the inequality\\ \scalebox{0.88}{$\abs{\scal{\Amatrix \x}{\vi}}\leq\abs{\scal{\Amatrix \x}{\y-\vi}}+\abs{\scal{\Amatrix \x}{\y}}$} together with the Cauchy--Schwarz inequality, it holds that {\bbl 
%\scalebox{0.88}{$
\begin{align*}
\underset{\vi\in\sNOne}{\sup}\abs{\scal{\Amatrix \x}{\vi}}=\normiiN{\Amatrix \x}&\leq \normiiN{\Amatrix \x}\normiiN{\y-\vi}+\underset{\y\in\n}{\sup}\abs{\scal{\Amatrix \x}{\y}}  \\
&\leq \frac{1}{2}\normiiN{\Amatrix \x}+\underset{\y\in\n}{\sup}\abs{\scal{\Amatrix \x}{\y}}.
\end{align*}
}
%$}. 
Thus, \scalebox{0.88}{$\normiiN{\Amatrix \x}\leq 2\underset{\y\in\n}{\sup}\abs{\scal{\Amatrix \x}{\y}}$}, which, combined with~\eqref{supsup1}, yields 
{\footnotesize{
\begin{equation}\label{sup4MN}
\norme{\Amatrix}\leq 4 \underset{\x\in\M,\, \y\in\n}{\sup} \abs{\scal{\Amatrix \x}{\y}}.
\end{equation}
}}
By Lemma~\ref{volumLemma}, \scalebox{0.88}{$\abs{\M}\leq 5^n$} and \scalebox{0.88}{$\abs{\n}\leq 5^N$}. Moreover, \scalebox{0.88}{$\scal{\Amatrix \x}{\y}=\frac{\hat{q}}{\sqrt{n}}\sum_{i=1}^{N}\sum_{j=1}^{n}\aij\Yij$}, where the coefficients $\aij:=x_j y_i$ satisfy \scalebox{0.88}{$\sum_{i=1}^{N}\sum_{j=1}^{n}{\aij}^2= \sum_{i=1}^{N}y_i^2\left(\sum_{j=1}^{n} x_j^2\right)=1$}, since both $\x$ and $\y$ are unit vectors. Thus, it follows from Proposition~\ref{HoeffdingTypeIneq} that for all $\x\in\M$ and $\y\in\n$, {\bbl 
\[
\prob{\abs{\scal{\Amatrix \x}{\y}}>\frac{1}{4}\hat{q}t\sqrt{\frac{N}{n}}}\leq 2e^{-\frac{c_2t^2}{16}N}
\qquad \forall t>0, \] 
}
%for all \scalebox{0.88}{$t>0$}, 
where \scalebox{0.88}{$c_2={\bl \left(\frac{c_{2,3}}{K_1}\right)^2}$} is the constant in Proposition~\ref{HoeffdingTypeIneq}.

{\bbl
Now consider the event $\mathscr{E}_1:=\accolade{\norme{\Amatrix}>\hat{q}t\sqrt{\frac{N}{n}}}.$ In order to bound $\prob{\mathscr{E}_1}$, we first observe from~\eqref{sup4MN} that $\mathscr{E}_1\subseteq \accolade{\underset{\x\in\M,\, \y\in\n}{\sup}\abs{\scal{\Amatrix \x}{\y}}>\frac{1}{4}\hat{q}t\sqrt{\frac{N}{n}}}:=\mathscr{E}_2$, and that the occurrence of $\mathscr{E}_2$ implies the existence of $\x\in\M$ and $\y\in\n$ such that $\abs{\scal{\Amatrix \x}{\y}}>\frac{1}{4}\hat{q}t\sqrt{\frac{N}{n}}$, that is, $\mathscr{E}_2\subseteq \accolade{\exists \x\in\M, \exists \y\in\n:\abs{\scal{\Amatrix \x}{\y}}>\frac{1}{4}\hat{q}t\sqrt{\frac{N}{n}}}:=\mathscr{E}_3$. Thus, using the latter inclusions and then} taking the union bound lead to {\bbl 
%the following bound on $P_1:=\prob{\norme{\Amatrix}>\hat{q}t\sqrt{\frac{N}{n}}}$:
{\small{
\begin{align*}
%\prob{\norme{\Amatrix}>\hat{q}t\sqrt{\frac{N}{n}}}
\prob{\norme{\Amatrix}>\hat{q}t\sqrt{\frac{N}{n}}}=\prob{\mathscr{E}_1}
&\leq 
    \prob{\mathscr{E}_2}
    %\\
%&
\leq \prob{\mathscr{E}_3}
\leq
    \ds{\sum_{\x\in\M} \sum_{\y\in\n}\prob{\abs{\scal{\Amatrix \x}{\y}}>\frac{1}{4}\hat{q}t\sqrt{\frac{N}{n}}}}\\
    &\leq 
    \ds{\abs{\n}\abs{\M}\underset{\x\in\M,\, \y\in\n}{\sup} \prob{\abs{\scal{\Amatrix \x}{\y}}>\frac{1}{4}\hat{q}t\sqrt{\frac{N}{n}}}}
    \\
&\leq 
     5^{2N}\cdot2e^{-\frac{c_2}{16}Nt^2}
    %\\
%&
=
    2e^{-\frac{c_2}{16}Nt^2 + 2N\ln 5}.
\end{align*}}}
}
%\scalebox{0.81}{$\prob{\norme{\Amatrix}>\hat{q}t\sqrt{\frac{N}{n}}}\leq \prob{\underset{\x\in\M,\, \y\in\n}{\sup}\abs{\scal{\Amatrix \x}{\y}}>\frac{1}{4}\hat{q}t\sqrt{\frac{N}{n}}}\leq \prob{\exists \x\in\M, \exists \y\in\n:\abs{\scal{\Amatrix \x}{\y}}>\frac{1}{4}\hat{q}t\sqrt{\frac{N}{n}}}\leq$}\\
%\scalebox{0.81}{$\ds{\sum_{\x\in\M} \sum_{\y\in\n}\prob{\abs{\scal{\Amatrix \x}{\y}}>\frac{1}{4}\hat{q}t\sqrt{\frac{N}{n}}} \leq \abs{\n}\abs{\M}\underset{\x\in\M,\, \y\in\n}{\sup} \prob{\abs{\scal{\Amatrix \x}{\y}}>\frac{1}{4}\hat{q}t\sqrt{\frac{N}{n}}}}\leq 5^{2N}\cdot2e^{-\frac{c_2}{16}Nt^2}=$}\\
%\scalebox{1}{$2e^{-\frac{c_2}{16}Nt^2 + 2N\ln 5}.$}
%\begin{eqnarray*}
%\prob{\norme{\Amatrix}>\hat{q}t\sqrt{\frac{N}{n}}}&\leq& \prob{\underset{\x\in\M,\, \y\in\n}{\sup}\abs{\scal{\Amatrix \x}{\y}}>\frac{1}{4}\hat{q}t\sqrt{\frac{N}{n}}}
%\leq \prob{\exists \x\in\M, \exists \y\in\n:\abs{\scal{\Amatrix \x}{\y}}>\frac{1}{4}\hat{q}t\sqrt{\frac{N}{n}}}  \\
%&\leq&\sum_{\x\in\M} \sum_{\y\in\n}\prob{\abs{\scal{\Amatrix \x}{\y}}>\frac{1}{4}\hat{q}t\sqrt{\frac{N}{n}}} \leq \abs{\n}\abs{\M}\underset{\x\in\M,\, \y\in\n}{\sup} \prob{\abs{\scal{\Amatrix \x}{\y}}>\frac{1}{4}\hat{q}t\sqrt{\frac{N}{n}}}\\
%&\leq& 5^{2N}\cdot2e^{-\frac{c_2}{16}Nt^2}=2e^{-\frac{c_2}{16}Nt^2 + 2N\ln 5}. 
%\end{eqnarray*}
Then the inequality
{\small{
\begin{equation}\label{forChoices}
\prob{\norme{\Amatrix}>\hat{q}t\sqrt{\frac{N}{n}}}\leq 2e^{-c_0^{\star}Nt^2}\ \forall\ t\geq C_0^{\star},\ \mbox{\normalsize for some constants}\ C_0, c_0^{\star}>0,
\end{equation}
}}
holds, provided \scalebox{0.88}{$(\frac{c_2}{16}-c_0^{\star})t^2\geq 2\ln 5$}. For \scalebox{0.88}{$c_0^{\star}:=\frac{c_2}{32}$}, then \scalebox{0.85}{$t\geq 8\left(\frac{\ln 5}{c_2}\right)^{1/2}={\bl \frac{1}{\tilde{c}_{2,3}} \sqrt{\frac{\ln 5}{\ln{\left(\frac{q}{2}+1\right)}}}}=:C_0$}, {\bl where \scalebox{0.88}{$\tilde{c}_{2,3}:=\frac{{c}_{2,3}}{8}$}.} Then by~\eqref{forChoices},  \scalebox{0.88}{$\prob{\norme{\Amatrix}>\hat{q}t\sqrt{\frac{N}{n}}}\leq e^{-\frac{c_0^{\star}}{2}Nt^2}\  \forall\ t\geq C_0,$} for \scalebox{0.88}{$N$} sufficiently large, and the proof is completed. 
\end{proof}
The largest singular value of a hashing-like matrix $\Hfrac$ given by~\eqref{HashlikeEnsemble}, which is also that of $\Amatrix$ provided by~\eqref{AhashingLike}, satisfies the following probabilistic bound.
\begin{corollary}\label{LargestSingCor}
Consider \scalebox{0.88}{$\Amatrix$} as defined in~\eqref{AhashingLike}.
{\bl There exists an absolute constant \scalebox{0.88}{$\tilde{c}_{2,3}>0$} such that} for \scalebox{0.88}{$N$} sufficiently large,
{\small{
\begin{equation}\label{largestSingValHashing}
\prob{s_1(\Amatrix)>t\sqrt{\frac{N}{s}}}\leq \exp\left(-{\bl \tilde{c}_{2,3}^2}\ln\left(\frac{n}{s}+1\right)N t^2 \right)\leq e^{-{\bl \left(\tilde{c}_{2,3}^2\ln 2\right)}N t^2},\  \forall\ t\geq {\bl C_0',}
\end{equation}
}}
{\bl where \scalebox{0.88}{$C_0'=\frac{1}{\tilde{c}_{2,3}}\sqrt{\frac{\ln 5}{\ln 2}}$}.}
\end{corollary}

\begin{proof}
The result is derived from that of Theorem~\ref{LargestSingResult} in the special case where \scalebox{0.88}{$\frac{1}{q}=\frac{s}{2n}$}. We first observe that 
{\bl \scalebox{0.88}{$C_0\leq \frac{1}{\tilde{c}_{2,3}}\sqrt{\frac{\ln 5}{\ln 2}}=:C_0'$} and that \scalebox{0.88}{$c_0\geq \tilde{c}_{2,3}^2\ln 2$}.}
Then~\eqref{largestSingValHashing} immediately follows from~\eqref{largestSingValBound} after observing that\\ \scalebox{0.88}{$\hat{q}\sqrt{\frac{N}{n}} =\sqrt{\frac{q}{2}}\sqrt{\frac{N}{n}}=\sqrt{\frac{n}{s}}\sqrt{\frac{N}{n}}= \sqrt{\frac{N}{s}}$}, and the proof is completed.
\end{proof}

To derive a bound on the smallest singular value in Theorem~\ref{SmallestSingResult} below, we need the following estimate for a small ball probability of a sum of independent random variables, which is crucial for the proof of Lemma~\ref{TensorLemma}.
% required to derive 
\begin{lemma}\label{smallBallLemma}(Small ball probability).
Let \scalebox{0.88}{$X_1, \dots, X_m$} be independent copies of \scalebox{0.88}{$X\sim \bbmu_X(q)$}. Let \scalebox{0.88}{$\bm{a}=(a_1, \dots, a_m)\in\mathbb{S}^{m-1}$} and define \scalebox{0.88}{$S_{X,m}:=\sum_{\ell=1}^m a_{\ell} X_{\ell}$}. Then 
{\footnotesize{
\begin{equation}\label{smallBallRes}
\prob{\abs{S_{X,m}}\leq \frac{1}{2}}\leq 1-\gamma^2, \quad\mbox{\normalsize where}\ \gamma=\gamma(q):=\frac{3}{8qK_4^2}=\frac{3\ln(\hat{q}^2+1)}{{\bl 16C_4^2}\hat{q}^2}\in (0,1).
\end{equation}
}}
\end{lemma}

\begin{proof}
By the same arguments used in Lemma~\ref{expoLambdaU} to demonstrate \scalebox{0.88}{$\E{T_1^2}=1$}, we first observe that \scalebox{0.88}{$\E{S_{X,m}^2}=1$}. By Khintchine~\eqref{Khintchine} or using~\eqref{T1Khintcine} with minor changes, \scalebox{0.88}{$\E{S_{X,m}^4}\leq (2\hat{q}K_4)^4=\left(\frac{3}{4\gamma}\right)^2$}. It follows from the Paley--Zygmund inequality (see, e.g.,~\cite[Exercise~1.1.9]{tao2012topicsInRMT}): \scalebox{0.88}{$\prob{W>\vartheta\E{W}}\geq (1-\vartheta)^2\cfrac{(\E{W})^2}{\E{W^2}}$} with \scalebox{0.88}{$W:=S_{X,m}^2$} and \scalebox{0.88}{$\vartheta=\frac{1}{4}$}, that \scalebox{0.88}{$\prob{S_{X,m}^2>\frac{1}{4}}=\prob{S_{X,m}^2>\frac{1}{4}\E{S_{X,m}^2}}\geq \frac{(3/4)^2}{\E{S_{X,m}^4}}\geq \gamma^2$}. Then \scalebox{0.88}{$\frac{3}{8qK_4^2}=\frac{3\ln(\hat{q}^2+1)}{{\bl 16C_4^2}\hat{q}^2}$} which follows from the definition of \scalebox{0.88}{$K_4$} given by~\eqref{Khintchine} after simple calculations{\bl . The proof is completed by noticing that \scalebox{0.88}{$\gamma\in (0,1)$} since \scalebox{0.88}{$C_4>1$}}.
\end{proof}

Lemma~\ref{smallBallLemma} implies the following invertibility estimate for a fixed vector.
\begin{lemma}\label{TensorLemma}(Invertibility for a fixed vector).
Let $\gamma$ be the same constant of Lemma~\ref{smallBallLemma}. Then the matrix \scalebox{0.88}{$\Amatrix\in \rNn$} in~\eqref{AMatrixEnsemble} satisfies 
{\footnotesize{
\begin{equation}\label{TensorIneq}
\prob{\normiiN{\Amatrix \x}\leq \frac{\gamma}{4}\sqrt{\frac{N}{n}}}\leq \psi^N \quad\forall\ \x\in \snOne,\ \mbox{\normalsize where}\ \psi=\psi(q)=e^{\gamma^2\left(\frac{\ln 2}{4}-\frac{1}{2}\right)}\in (0, 1).
\end{equation}
}}
\end{lemma}

\begin{proof}
The proof uses a {\it tensorization argument} (see, e.g.,~\cite[Lemma~2.2-(2)]{RudVersh2008Lit}), which was also used in the proof of~\cite[Proposition~3.4]{LitPajRud2005Pol}. More precisely, define \scalebox{0.88}{$\xi_i=\abs{\sqrt{n}(\Amatrix\x)_i}=\abs{\sum_{j=1}^n x_j\Xij}$} for all \scalebox{0.88}{$i\in \intbracket{1, N}$}. By Lemma~\ref{smallBallLemma}, \scalebox{0.88}{$\prob{\xi_i^2\leq \lambda^2}\leq 1-\gamma^2$} with \scalebox{0.88}{$\lambda=\frac{1}{2}$}, which leads to \scalebox{0.88}{$\prob{\sum_{i=1}^N\xi_i^2\leq \frac{\gamma^2}{16}N}\leq \psi^N$} using the tensorization argument and hence yielding the required result, as detailed next. For any \scalebox{0.88}{$t, \tau>0$}, 
{\footnotesize{
\begin{equation}\label{expoTauNIneq}
\begin{split}
\prob{\normiiN{\Amatrix \x}\leq t\sqrt{\frac{N}{n}}} &= \prob{n\normiiN{\Amatrix \x}^2\leq t^2N}=\prob{\sum_{i=1}^N\xi_i^2\leq t^2N}\\
&=\prob{N-\frac{1}{t^2}\sum_{i=1}^N\xi_i^2\geq 0}\leq \E{\exp\left(\tau N - \frac{\tau}{t^2}\sum_{i=1}^N\xi_i^2\right)}\\
&=e^{\tau N}\prod_{i=1}^{N}\E{e^{-\tau \xi_i^2/t^2}},
\end{split}
\end{equation}
}}where we used Markov's inequality while the last equality results from the independence of the $\xi_i$. In order to bound the last expectation in~\eqref{expoTauNIneq}, we fix \scalebox{0.88}{$i\in \intbracket{1, N}$}, define \scalebox{0.88}{$\lambda$} as above, and  let \scalebox{0.88}{$t, \tau>0$}  be chosen later. We first observe that the inequality \scalebox{0.88}{$\prob{\xi_i^2\geq \lambda^2}\geq\gamma^2$} implies \scalebox{0.8}{$\prob{e^{-\tau \xi_i^2/t^2}>e^{-\tau\lambda^2/t^2}}\leq 1-\gamma^2$} and hence \scalebox{0.85}{$\prob{e^{-\tau \xi_i^2/t^2}>r}\leq 1-\gamma^2$} for all \scalebox{0.88}{$r\geq e^{-\tau\lambda^2/t^2}$}. Thus, using the integral identity~\cite[Lemma~1.2.1]{vershynin2018HDP} and the fact that the support of the random variable \scalebox{0.88}{$e^{-\tau \xi_i^2/t^2}$} is \scalebox{0.88}{$(0, 1]$}, it holds that
{\footnotesize{
\begin{equation}\label{TauExpect}
\begin{split}
\E{e^{-\tau \xi_i^2/t^2}} &=\int_{0}^{+\infty}\prob{e^{-\tau \xi_i^2/t^2}>r} dr=\int_{0}^{1}\prob{e^{-\tau \xi_i^2/t^2}>r} dr\\
&= \int_{0}^{e^{-\tau\lambda^2/t^2}}\prob{e^{-\tau \xi_i^2/t^2}>r} dr + \int_{e^{-\tau\lambda^2/t^2}}^{1}\prob{e^{-\tau \xi_i^2/t^2}>r} dr\\
&\leq e^{-\tau\lambda^2/t^2} + (1-\gamma^2)\left(1-e^{-\tau\lambda^2/t^2}\right).
\end{split}
\end{equation}
}}Let $\tau:=\alpha t^2/\lambda^2=4\alpha t^2$ with $t:=\frac{\gamma}{4}$ and $\alpha=\ln 2$. It follows from~\eqref{expoTauNIneq} and~\eqref{TauExpect} that 
{\small{
\begin{align*}
\prob{\normiiN{\Amatrix \x}\leq \frac{\gamma}{4}\sqrt{\frac{N}{n}}}
&\leq e^{4\alpha t^2 N}\left(1-\gamma^2\left(1-e^{-\alpha}\right)\right)^N
\\
&\leq \left(e^{4\alpha t^2-\gamma^2\left(1-e^{-\alpha}\right)}\right)^N=:\psi^N,
\end{align*}}}
where the last inequality follows from $1-\mathfrak{u}\leq e^{-\mathfrak{u}}$ with $\mathfrak{u}=\gamma^2\left(1-e^{-\alpha}\right)$. Then, the proof is completed by observing that $4\alpha t^2-\gamma^2\left(1-e^{-\alpha}\right)=\gamma^2\left(\frac{\ln 2}{4}-\frac{1}{2}\right)<0$.
\end{proof}

\begin{remark}\label{epsilonRemark}
Consider \scalebox{0.88}{$\gamma={\bl\frac{3\ln\left(\frac{q}{2}+1\right)}{ 8C_4^2q}\in (0,1)}$} and \scalebox{0.88}{$C_0={\bl \frac{1}{\tilde{c}_{2,3}}\sqrt{\frac{\ln 5}{\ln{\left(\frac{q}{2}+1\right)}}}}$} of Lemma~\ref{smallBallLemma} and Theorem~\ref{LargestSingResult}, respectively.  
Let \scalebox{0.88}{$\kappa_0:={\bl \frac{3\tilde{c}_{2,3}}{128C_4^2\sqrt{\ln 5}}}$} and \scalebox{0.88}{$\epsilon_q:=\frac{\gamma}{8C_0\hat{q}}=\frac{2\kappa_0\left(\ln\left(\frac{q}{2}+1\right)\right)^{3/2}}{q\hat{q}}=$} \scalebox{0.88}{$\kappa_0\left(\frac{2}{q}\ln\left(\frac{q}{2}+1\right)\right)^{3/2}<1$}. Recalling \scalebox{0.88}{$\psi=e^{-\kappa_0'\gamma^2}$} of Lemma~\ref{TensorLemma} where \scalebox{0.88}{$\kappa_0'=\frac{1}{2}-\frac{\ln 2}{4}>0$}, we observe that \scalebox{0.88}{$\frac{1}{3}\epsilon_q e^{\kappa_0'\gamma^2}<\frac{1}{3}\kappa_0e^{\kappa_0'}\left(\frac{2}{q}\ln\left(\frac{q}{2}+1\right)\right)^{3/2}< 1$}, which implies that \scalebox{0.88}{$\epsilon_q<3e^{-\kappa_0'\gamma^2}=3\psi$}, whence \scalebox{0.88}{$\epsilon_q<\min\accolade{3\psi,\ 1}$}.
\end{remark}

Combining Lemma~\ref{TensorLemma} with the $\varepsilon$-net argument leads to the following estimation of $s_n(\Amatrix)$ with significantly different dimensions of $\Amatrix$, as was also the case in~\cite[Proposition~4.7]{Rudelson2014LectNotes}.
\begin{theorem}\label{SmallestSingResult}(Smallest singular value).
Consider the matrix \scalebox{0.85}{$\Amatrix\in \rNn$} of~\eqref{AMatrixEnsemble} with \scalebox{0.88}{$N$} sufficiently large. Let \scalebox{0.88}{$\gamma, \psi\in (0,1)$} be the same constants of Lemma~\ref{TensorLemma}, and let \scalebox{0.88}{$\epsilon_q\in (0,1)$} be defined as in Remark~\ref{epsilonRemark}. Fix \scalebox{0.88}{$\kappa\in (0, 1)$}, and assume that \scalebox{0.88}{$n\leq \nu_0 N$}, where \scalebox{0.88}{$\nu_0\in (0, \alpha_0)$} with \scalebox{0.88}{$\alpha_0:=\dfrac{\kappa\ln(1/\psi)}{\ln(3/\epsilon_q)}\in (0, 1)$}. Then, 
{\footnotesize{
\begin{equation}\label{smallestResult}
\prob{s_n(\Amatrix)\leq \kappa_1\sqrt{\frac{N}{n}}}\leq e^{-\kappa_2 N},
\end{equation}
}}
where \scalebox{0.88}{$\kappa_1:=\kappa_1(q)=\frac{\gamma}{8}$} and \scalebox{0.88}{$\kappa_2:=\kappa_2(q)=\frac{1}{4}(1-\kappa)\left(1-\frac{\ln 2}{2}\right)\gamma^2$}.
\end{theorem}

\begin{proof}
The proof is inspired by that of~\cite[Proposition~4.7]{Rudelson2014LectNotes}. Let \scalebox{0.88}{$\n\subset\snOne$} be an \scalebox{0.88}{$\epsilon_q$}-net with cardinality \scalebox{0.88}{$\abs{\n}\leq \left(\frac{3}{\epsilon_q}\right)^n$} thanks to Lemma~\ref{volumLemma}. By Lemma~\ref{TensorLemma},
{\footnotesize{
\begin{equation}\label{unionBoundIneq}
\prob{\exists \y\in\n\,: \normiiN{\Amatrix \y}\leq \frac{\gamma}{4}\sqrt{\frac{N}{n}}}\leq\sum_{\y\in \n}\prob{\normiiN{\Amatrix \y}\leq \frac{\gamma}{4}\sqrt{\frac{N}{n}}}\leq \psi^N\left(\frac{3}{\epsilon_q}\right)^n.
\end{equation}
}}
Consider the event \scalebox{0.88}{$\mathscr{E}:=\accolade{\norme{\Amatrix}\leq \hat{q}C_0\sqrt{\frac{N}{n}}}\cap\accolade{\normiiN{\Amatrix \y}>\frac{\gamma}{4}\sqrt{\frac{N}{n}}\ \forall\y\in\n}$}, where \scalebox{0.88}{$C_0$} is the constant of Theorem~\ref{LargestSingResult}. Since \scalebox{0.88}{$N$} is sufficiently large,~\eqref{unionBoundIneq} and Theorem~\ref{LargestSingResult} yield
{\footnotesize{
\begin{equation}\label{VcIneq}
\prob{\mathscr{E}^c}\leq\prob{\norme{\Amatrix}>\hat{q}C_0\sqrt{\frac{N}{n}}}+\prob{\exists\y\in\n\,: \normiin{\Amatrix\y}\leq \frac{\gamma}{4}\sqrt{\frac{N}{n}}}\leq e^{-c_0C_0^2 N}+\left(\psi\left(\frac{3}{\epsilon_q}\right)^{\frac{n}{N}}\right)^N,
\end{equation}
}}
where $c_0$ is the constant of Theorem~\ref{LargestSingResult} satisfying \scalebox{0.88}{$c_0C_0^2=\ln 5$}. We observe that \scalebox{0.88}{$\alpha_0=\frac{\kappa\ln(1/\psi)}{\ln(3/\epsilon_q)}<\kappa<1$} since \scalebox{0.88}{$\epsilon_q<3\psi$} by Remark~\ref{epsilonRemark}, and the inequality \scalebox{0.88}{$\frac{n}{N}< \alpha_0=\frac{\ln(\psi^{-\kappa})}{\ln(3/\epsilon_q)}$} yields \scalebox{0.88}{$\psi\left(\frac{3}{\epsilon_q}\right)^{\frac{n}{N}}=\psi e^{\frac{n}{N}\ln (3/\epsilon_q)}<e^{-\ln(\psi^{\kappa-1})}<1$}. Thus, since \scalebox{0.88}{$N$} is sufficiently large, it follows from~\eqref{VcIneq} that \scalebox{0.88}{$\prob{\mathscr{E}^c}\leq e^{-(\ln 5)N}+e^{-\ln(\psi^{\kappa-1})N}\leq e^{-c_3 N}$} for some \scalebox{0.88}{$c_3>0$}, provided that \\ \scalebox{0.88}{$c_3 <\min\accolade{\ln(\psi^{\kappa-1}),\ \ln 5}$}. Recalling \scalebox{0.88}{$\psi=e^{\gamma^2\left(\frac{\ln 2}{4}-\frac{1}{2}\right)}$}, we choose \scalebox{0.88}{$c_3:=\frac{1}{2}\ln(\psi^{\kappa-1})=\kappa_2<\ln 5$}, where we used the fact that 
{\bl \scalebox{0.88}{$\kappa_2=\frac{1}{4}\gamma^2(1-\kappa)\left(1-\frac{\ln 2}{2}\right)<\frac{1}{4}\left(1-\frac{\ln 2}{2}\right)<\ln 5$}.}

Assume that \scalebox{0.88}{$\mathscr{E}$} occurs. \scalebox{0.88}{$\forall\x\in \snOne,\ \exists \y\in\n$} such that \scalebox{0.88}{$\normiin{\y-\x}<\epsilon_q=\frac{\gamma}{8C_0\hat{q}}$}. Thus, 
\scalebox{0.88}{$\normiiN{\Amatrix \x}=\normiiN{\Amatrix [\y-(\y-\x)]}\geq$} \scalebox{0.88}{$\normiiN{\Amatrix \y}-\norme{\Amatrix}\normiin{\y-\x}>\frac{\gamma}{4}\sqrt{\frac{N}{n}}-C_0\hat{q}\epsilon_q\sqrt{\frac{N}{n}}=\frac{\gamma}{8}\sqrt{\frac{N}{n}},$} whence \scalebox{0.88}{$\mathscr{E}\subseteq\accolade{\normiiN{\Amatrix \x}>\kappa_1\sqrt{\frac{N}{n}}\ \forall\x\in\snOne}=:\mathscr{F}$}.\\ Since \scalebox{0.88}{$\accolade{\underset{\x\in\snOne}{\min}\normiiN{\Amatrix\x}\leq \kappa_1\sqrt{\frac{N}{n}}}\subseteq \mathscr{F}^c$}, then \scalebox{0.88}{$\prob{s_n(\Amatrix) \leq \kappa_1\sqrt{\frac{N}{n}}}\leq \prob{\mathscr{F}^c}\leq \prob{\mathscr{E}^c}\leq e^{-\kappa_2 N}$}, which achieves the proof.
\end{proof}

The result of Theorem~\ref{SmallestSingResult} shows that the hashing parameter~$s$ can be chosen so that the smallest singular value remains bounded below by a fixed quantity with high probability, as detailed next.

\begin{remark}\label{smalSingRemark1}
Assume that the entries \scalebox{0.88}{$\Xij$} defining \scalebox{0.88}{$\Amatrix$} are distributed as in Corollary~\ref{LargestSingCor}, where for any \scalebox{0.88}{$n\geq 1$}, \scalebox{0.88}{$s=s(n)=\frac{2n}{q}$} for some fixed $q\geq 2$. Then it holds that \scalebox{0.88}{$\prob{s_n(\Amatrix)\geq \kappa_1(q)}\geq \prob{s_n(\Amatrix)\geq \kappa_1(q)\sqrt{\frac{N}{n}}}=1-\mathcal{O}\left(N^{-\kappa_2(q)}\right)$}.
\end{remark}

\section{Conclusion}
Inspired by $s$-hashing matrices with exactly $s$ nonzero elements on each column, this manuscript introduces {\bl an ensemble of sparse matrices encompassing so-called} $s$-hashing-like matrices with sub-Gaussian entries, whose expected number of nonzero elements on each column is~$s$. Using properties of sub-Gaussian and subexponential random variables, we demonstrate such matrices  to be Johnson--Lindenstrauss transforms (JLTs), and analyses of their extreme singular values have been carried out, which are largely facilitated by the knowledge of the exact distribution of their independent entries. While  many  works exist on the extreme singular values of random matrices with sub-Gaussian entries,  none, to our knowledge,  is specific to either $s$-hashing or $s$-hashing-like matrices. 
%Indeed, the main results of this manuscript are clearly stated in terms of the parameter $s$, the dimensions of the matrices, and known constants, unlike what is generally the case in random matrix theory where constants are often unknown and assumed to be absolute. 
%On the other hand, numerical experiments yield similar results for the $s$-hashing and $s$-hashing-like matrices and also suggest better bounds on $s$, the extreme singular values, and so forth, than are predicted by the theoretical analyses.

%Future work may focus on obtaining sharper bounds than those derived in this manuscript; and 
Since $s$-hashing matrices are already shown to be JLTs, future work may focus on theoretical analyses of their extreme singular values.

\subsection*{Acknowledgments}
This material was based upon work supported by the U.S.\ Department of
Energy, Office of Science, Office of Advanced Scientific Computing
Research, applied mathematics and SciDAC programs under Contract Nos.\
DE-AC02-05CH11231 and DE-AC02-06CH11357.

\bibliography{hashing-bib}
\bibliographystyle{abbrvnat} 

%\bibliographystyle{siamplain}
%\bibliography{hashing-bib}

{\bl
\section*{Appendix}
In the experiments simulating the results of Corollary~\ref{CorBaiYin}, the independent entries \scalebox{0.88}{$\Hij:=\frac{1}{\sqrt{s}}\Yij$}, with  \scalebox{0.88}{$\prob{\Yij=\pm 1}=$} \scalebox{0.88}{$\frac{s}{2n}=\frac{1}{q}$} and \scalebox{0.88}{$\prob{\Yij= 0}=1-\frac{s}{n}=1-\frac{2}{q}$}, of $s$-hashing-like matrices \scalebox{0.88}{$\Hfrac$} are generated as follows. Let \scalebox{0.88}{$\Uij\sim\mathcal{U}(0,1)$} be a uniform random variable. Then, \scalebox{0.88}{$\Yij=-1$} if \scalebox{0.88}{$\Uij\in \left[\left.0, \frac{s}{2n}\right)\right.$}, \scalebox{0.88}{$\Yij=1$} if \scalebox{0.88}{$\Uij\in \left[\left.\frac{s}{2n}, \frac{s}{n} \right)\right.$}, and \scalebox{0.88}{$\Yij=0$} if \scalebox{0.88}{$\Uij\in \left[\frac{s}{n}, 1 \right]$}, so that \scalebox{0.88}{$\Yij$} is distributed as above.
}

To simulate the results of Corollary~\ref{CorBaiYin}, we consider \scalebox{0.88}{$\accolade{N_{\ell}}_{\ell\leq 100}$}, \scalebox{0.88}{$\accolade{n_{\ell}}_{\ell\leq 100}$} and \scalebox{0.88}{$\accolade{s_{\ell}}_{\ell\leq 100}$} defined as follows. The increasing terms of \scalebox{0.85}{$\accolade{N_{\ell}}_{\ell\leq 100}:=500,\ \, 527,\ \,  556,\ \, 587,\ \, 619,$} \scalebox{0.85}{$653,\  \dots,\  94789,\ \, 89849,\ \,  10^5$} are obtained by rounding \scalebox{0.88}{$100$} numbers (between \scalebox{0.88}{$500$} and \scalebox{0.88}{$10^5$}) arranged on a log scale, using \scalebox{0.9}{${\sf round(logspace(log10(500),\ log10(1e5),\ 100))}$}, a {\sf MATLAB} command. Then for each \scalebox{0.88}{$\ell\in \intbracket{1,100}$}, \scalebox{0.88}{$n_{\ell}$} is obtained by rounding \scalebox{0.88}{$N_{\ell}/100$} while~\scalebox{0.88}{$s_{\ell}$} is computed by rounding \scalebox{0.88}{$n_{\ell}/5$} so that the aspect ratio \scalebox{0.88}{$\frac{n_{\ell}}{N_{\ell}}\to \mathfrak{y}:=\frac{1}{100}$} and \scalebox{0.88}{$\frac{s_{\ell}}{n_{\ell}}\approx \frac{1}{5}$} as $\ell$ gets larger. Consider the $s_{\ell}$-hashing-like matrices \scalebox{0.88}{$\Amatrix_{\ell}=\Hfrac_{\ell}^{\top}\in \R^{N_{\ell} \times n_{\ell}}$} for \scalebox{0.88}{$\ell = 1,\dots,100$}. 
 
For each \scalebox{0.88}{$\ell$}, realizations \scalebox{0.88}{$s_{1,\ell}^{1},\dots,s_{1,\ell}^{100}$} of \scalebox{0.88}{$s_1(\Amatrix_{\ell})$}, and \scalebox{0.88}{$s_{n,\ell}^{1},\dots,s_{n,\ell}^{100}$} of \scalebox{0.88}{$s_n(\Amatrix_{\ell})$} are generated, as well as \scalebox{0.88}{$\ m_{1,\ell}$} and \scalebox{0.88}{$\ M_{1,\ell}$} defined respectively by \scalebox{0.88}{$\ m_{1,\ell}:=\min\accolade{s_{1,\ell}^{k},\ k\in \intbracket{1,100}}$} and \scalebox{0.88}{$\ M_{1,\ell}:=\max\accolade{s_{1,\ell}^{k},\ k\in \intbracket{1,100}}$}, and the average or sample mean \scalebox{0.88}{$\bar{s}_{1,\ell}:=\frac{1}{100}\sum_{k=1}^{100}s_{1,\ell}^{k}$}; and \scalebox{0.88}{$m_{n,\ell}$}, \scalebox{0.88}{$M_{n,\ell}$} and \scalebox{0.88}{$\bar{s}_{n,\ell}$} are defined analogously. Figure~\ref{lab01} shows on the one hand the graphs of \scalebox{0.88}{$m_{1,\ell}$}, \scalebox{0.88}{$M_{1,\ell}$}, and \scalebox{0.88}{$\bar{s}_{1,\ell}$} with respect to \scalebox{0.88}{$N_{\ell}$}, for \scalebox{0.88}{$\ell=1,\dots,100$}, and on the other hand similar graphs for the smallest singular values, using \scalebox{0.88}{$m_{n,\ell}$}, \scalebox{0.88}{$M_{n,\ell}$}, and \scalebox{0.88}{$\bar{s}_{n,\ell}$}. These graphs confirm the predictions of the theory that (almost surely) \scalebox{0.88}{$s_n(\Amatrix_{\ell})\to \frac{1}{\sqrt{\mathfrak{y}}}-1=9$} and \scalebox{0.88}{$s_1(\Amatrix_{\ell})\to \frac{1}{\sqrt{\mathfrak{y}}}+1=11$} as \scalebox{0.88}{$\ell$} gets larger.

% This can be deleted in the version that appears, it is only needed for the Argonne PANDA submission and ArXiv submission:
\clearpage
\framebox{\parbox{\linewidth}{
The submitted manuscript has been created by UChicago Argonne, LLC, Operator of 
Argonne National Laboratory (``Argonne''). Argonne, a U.S.\ Department of 
Energy Office of Science laboratory, is operated under Contract No.\ 
DE-AC02-06CH11357. 
The U.S.\ Government retains for itself, and others acting on its behalf, a 
paid-up nonexclusive, irrevocable worldwide license in said article to 
reproduce, prepare derivative works, distribute copies to the public, and 
perform publicly and display publicly, by or on behalf of the Government.  The 
Department of Energy will provide public access to these results of federally 
sponsored research in accordance with the DOE Public Access Plan. 
http://energy.gov/downloads/doe-public-access-plan.}}

\end{document}